\documentclass[12pt]{article}
\usepackage{amsmath,amsfonts,amssymb,amsthm}
\usepackage{verbatim}
\usepackage{latexsym}
\newtheorem{theorem}{Theorem}[section]
\newtheorem{lemma}[theorem]{Lemma}
\newtheorem{proposition}[theorem]{Proposition}

\newtheorem{corollary}[theorem]{Corollary}

\newcommand\RR{{\Bbb R}}

\newcommand\ZZ{{\Bbb Z}}

\newcommand\R{{\Bbb R}}
\newcommand\C{{\Bbb C}}

\newcommand\Z{{\Bbb Z}}

\newcommand\supp{{\rm supp}}

\newcommand{\edth}{\eth}
\newcommand{\bedth}{\overline{\eth}}

\begin{document}
\title{Nearly tight frames of spin wavelets on the sphere}
\author{Daryl Geller\\
\footnotesize{Department of Mathematics, Stony Brook University}\\
\footnotesize\texttt{{daryl@math.sunysb.edu}}\\ \\
Azita Mayeli\\
\footnotesize{Department of Mathematics, Stony Brook University}\\
\footnotesize\texttt{{amayeli@math.sunysb.edu}}}
%\keywords{...}
\maketitle

\begin{abstract}
We show that the{\em \ spin wavelets} on the sphere $S^2$, which were constructed by the 
first author and Marinucci in \cite{GeMari08}, can be chosen so as to form a nearly tight frame.
These spin wavelets can be applied
(see \cite{spinphys} and \cite{spinstat}) to the study of the polarization of cosmic microwave
background radiation.  For certain of these frames, there is a positive $C$ such that
each frame element at scale $a^j$ is supported in a geodesic ball of radius $Ca^j$.
\end{abstract}

\section{Introduction}
In \cite{gmcw} and \cite{gmfr}, we have constructed nearly tight frames on any smooth compact Riemannian manifold 
${\bf M}$.  (Loosely, we call a frame ``nearly tight'' if by adjusting spacing and scaling on the space and frequency side, the 
frame bounds can be made as close to $1$ as desired.)  As we shall review, we showed (in the course of the proof
of Lemma 4.1 of \cite{gmcw}) that we could even arrange, that each frame element at scale $a^j$ be supported
in a geodesic ball of radius $a^j$.

The important special case in which ${\bf M}$ is a sphere was dealt with earlier by
Narcowich, Petrushev and Ward, in \cite{narc1}, \cite{narc2}.  Our philosophy in constructing frames
is quite similar to theirs.
There were, however, these important differences
between our work and theirs: First, their frames (in \cite{narc1}, \cite{narc2})
were always finite linear combinations of spherical harmonics,
hence real analytic, so they could not be supported in a small ball.  However, they were able to construct
frames that were actually tight, by performing exact discretizations of integrals, by using appropriate
cubature points.

In this article, building on results of the first author and Marinucci \cite{GeMari08},
we shall generalize results from our articles \cite{gmcw} and \cite{gmfr}, to construct nearly tight frames
of spin wavelets on the sphere.  As we shall explain, spin functions on the sphere are sections of a complex line
bundle, which is of great importance in physics.  We will be able to arrange, for some $C > 0$, that
each frame element at scale $a^j$ is supported in a geodesic ball of radius $Ca^j$.

The main immediate motivation for this work is application to the study of cosmic microwave background radiation
(CMB).  This
radiation, emitted only 370,000 years after the Big Bang, has been called the ``Rosetta Stone'' of the physics of the
early universe.  It has both a temperature and a polarization; the former is a scalar quantity, while the latter
is naturally viewed as a spin $2$ quantity.  Physical and statistical consequences of the
theory of spin wavelets, as presented in \cite{GeMari08} and in this article, are discussed in
\cite{spinphys} and \cite{spinstat}.  

In this introduction, we will first review spin functions, then review our philosophy for constructing nearly
tight frames on manifolds.  Finally we will state our main results, and explain the plan of the paper.

We refer the reader to \cite{gmcw} for a discussion of earlier work on frames on manifolds, and
to \cite{spinphys} and \cite{spinstat} for references to the extensive physical and statistical literature
related to CMB.

We would like to thank Domenico Marinucci for introducing us to this area of research, and for
many helpful discussions.

\subsection{Spin functions and kernels}

We first review a number of facts about spin $s$ functions and kernels.
For proofs, see \cite{GeMari08}.

Let $U_I$ denote $S^2\setminus\{{\bf N},{\bf S}\}$, where ${\bf N}=( 0, 0, 1)$ and ${\bf S}=( 0, 0, -1)$,
the north and south poles of the sphere respectively. For any
rotation $R\in SO(3)$ let $U_R= RU_I$. 
On $U_I$ we use standard spherical coordinates $(\theta,\phi)$
($0 < \theta < \pi$, $-\pi \leq \phi < \pi$), and analogously, on any $U_R$ we use coordinates
$(\theta_R,\phi_R)$ obtained by rotation of the coordinate system on $U_I$.
At each point $p$ of $U_R$ we let $\rho_R(p)$ be the unit tangent vector at $p$
which is tangent to the circle $\theta_R = $ constant, pointing in the direction of increasing $\phi_R$.
(Thus, for any $p \in U_I$, $\rho_R(Rp) = R_{*p} [\rho_{I}(p)]$. ($ R_{*p}$ is the induced map on tangent vectors at $p$.)) 
If $p \in U_{R_1} \cap U_{R_2}$,
it makes sense to define the (oriented)
angle from the  tangent vector $\rho_{R_1}(p)$ to  $\rho_{R_2}(p)$,
for $p$ in the intersection of $U_{R_1}$ and $U_{R_2}$. We let
$\psi_{p R_2 R_1}$ denote this angle.  
(See \cite{GeMari08} for a careful discussion of which is the oriented angle.)
Note that the intersection
of the two preceding sets consists of the entire sphere, excluding the points 
$\{R_1({\bf N}), R_1({\bf S}), R_2({\bf N}), R_2({\bf S})\}$.)\\
Let $s$ be an integer  and let $\Omega$ be an open subset of $S^2$.
We define $C_s^\infty(\Omega)$, the set of all smooth spin $s$
functions over $\Omega$, to be the set of 
$f=(f_R)_{R\in SO(3)}$ with $f_R\in C^\infty(\Omega\cap
U_R)$ for all $R$, such that for any $R_1, R_2 \in SO(3)$ and all $p\in U_{R_1}\cap
U_{R_2}\cap \Omega$ the components of $f$ satisfy the following ``spin relation'':
\begin{align}\label{spin-relation}
f_{R_2}(p)= e^{is\psi} f_{R_1}(p)
\end{align}
where $\psi:=\psi_{p R_2 R_1}$, the angle
introduced above. 
Observe
that for $s=0$, $C_s^\infty(\Omega)$ can be identified with
$C^\infty(\Omega)$, the set of smooth scalar functions over $\Omega$.

If, say, $R_1 = I$, $R_2 = R$, then heuristically
$f_R$ is $f_I$ ``looked at after the coordinates have been rotated by $R$'';
at $p$, it has been multiplied by $e^{is\psi}$.  

For any $s$, we can identify $C^{\infty}_s(\Omega)$ with the
sections over $\Omega$ of the complex line bundle ${\bf L}^s$ 
obtained by using the $e^{is\psi_{p R_2 R_1}}$ as transition functions from the
chart $U_{R_1}$ to the chart $U_{R_2}$; then $f_R$ is the trivialization of the section over $U_R \cap \Omega$.
However, in this article, except in  Proposition \ref{strsp}, we shall look
at spin functions as collections of functions $f=(f_R)_{R\in SO(3)}$ as above, since it is
conceptually easier to do so.
 
Analogously to $C^{\infty}_s$,  we define $C_s^\infty(S^2\times S^2)$, the set of
spin $s$ smooth kernels over $S^2$, to be the set of $K=(K_{R',R})_{R', R\in SO(3)}$ such
that $K_{R',R}\in C^\infty(U_{R_1'}\times U_R)$ and
\begin{align}\label{kernel-spin-relation}
K_{R_1', R_1}(p,q)= e^{is(\psi'-\psi)}K_{R', R}(p,q)~~~~~\forall~ p\in U_{R_1'}\cap U_{R'},~~ \forall ~q\in  U_{R_1}\cap U_{R}
\end{align}
where $\psi':=\psi_{p R_1' R'}$ and $\psi:=\psi_{q R_1 R}$. 
We may then naturally define a kernel operator ${\cal K}:
C_s^\infty(S^2)\rightarrow C_s^\infty(S^2)$ by $({\cal K}f)_{R'}(x)
= \int_{S^2} K_{R',R}(x,y) f_R(y) dy$, for all $x \in U_{R'}$. \\

 In a manner similar to the definition of  spin $s$ smooth function,
 we can define $L_s^2(\Omega)$ as the set of all
 $f=(f_R)_{R\in SO(3)}$ such that $f_R\in
 L^2(\Omega\cap U_R)$  for any $R\in SO(3)$, and such that (\ref{spin-relation})
holds a.e.\ on $\Omega$. For any $f, g\in
 L_s^2(\Omega)$, we may define  $\langle f, g\rangle= \langle f_R,
 g_R\rangle$, since the scalar product is independent of choice of
 $R$ (see (\ref{spin-relation})). $L_s^2(\Omega)$ is 
 evidently a Hilbert space with this scalar product. \\

 There is a unitary action of $SO(3)$ on $L_s^2(S^2)$, 
 given by  $f\rightarrow f^R$ (for $R \in SO(3)$), which is determined by
 \begin{align}
(f^R)_I(p)= f_R(Rp)~~~~ p\in U_I.
 \end{align}
Hence, by (\ref{spin-relation}), for any $R_0\in SO(3)$ and $p\in
U_I\cap U_{R_0}$ we have  $(f^R)_{R_0}(p)= e^{i
s\psi}f_R(Rp)$ where $\psi:=\psi_{p R_0 I}$. We shall think of
$f^R$ as a {\it rotate} of $f$. Observe that for $s=0$ we have the natural unitary action of $SO(3)$ on $L^2(S^2)$ and $f^R(p)=f(Rp)$. \\

We recall the  definition of {\it spin-raising}  and {\it
spin-lowering}
operators $\edth$ and $\bar\edth$ from \cite{np}, which are as follows:\\
For a spin  $s$ smooth function $f$,  define
\[ \edth_{sR} f_{R} = -(\sin \theta_R)^s \left(\frac{\partial}{\partial \theta_R} +
\frac{i}{\sin \theta_R}\frac{\partial}{\partial \phi_R} \right)(\sin
\theta_R)^{-s} f_{R}. \]

The ``spin-raising'' operator
$\edth: C^{\infty}_s(\Omega) \to C^{\infty}_{s+1}(\Omega)$ given by
$(\edth f)_R = \edth_{sR}f_R$ is well-defined. There is also a 
well-defined spin-lowering
operator $\bedth: C^{\infty}_s(\Omega) \to C^{\infty}_{s-1}(\Omega)$,
given by   $(\bedth f)_R = \bedth_{sR}f_R$, where
 $\bedth_{sR} f_R =
\overline{\edth_{(-s)R} \overline{f_R}}$.  The operators  $\edth$ and
$\bedth$ commute with the action of $SO(3)$ on smooth spin
functions.\\

Let $\{Y_{lm}: l \geq 0,\: -l \leq m \leq l\}$ be the standard basis
of spherical harmonics on $S^2$. For
$l \geq |s|$, we define the spin $s$ spherical harmonics of \cite{np} as follows: If $s \geq 0$,
${}_sY_{lm} = \edth^s Y_{lm}/b_{ls}$, where
$(\edth^s Y_{lm})_R(p)=\edth_{s-1,R}\cdots\edth_{1R}\edth_{0R}   Y_{lm}(p)$ ($p\in U_R$)  and $b_{ls} =
[(l+s)!/(l-s)!]^{1/2}$.  Similarly, if $s < 0$, we define ${}_sY_{lm} =
(-\bedth)^{|s|} Y_{lm}/b_{ls}$ for $b_{ls} = [(l-s)!/(l+s)!]^{1/2}$.
Then $\{{}_sY_{lm}: l
\geq |s|, -l \leq m \leq l\}$   forms an orthonormal basis for
$L^2_s(S^2)$. 

Analogously to the spherical Laplacian, for spin $s$ functions we
define the spin $s$ spherical Laplacian $\Delta_s$ as $-\bar\edth \edth$
for $s\geq 0$ and $-\edth\bar\edth$ for $s<0$. Then $\Delta_0$ is
the usual spherical Laplacian. For $l \geq |s|$, let ${\cal H}_{ls}
= <{}_sY_{lm}: -l \leq m \leq l>$.  Also, for $l \geq |s|$, let
$\lambda_{ls} = (l-s)(l+s+1)$ if $s \geq 0$, and let $\lambda_{ls} =
(l+s)(l-s+1)$ if $s < 0$. Then ${\cal H}_{ls}$ is the subspace of
$C^{\infty}_s(S^2)$ which consists of eigenfunctions of $\Delta_s$
for the eigenvalue $\lambda_{ls}$.

\subsection{Nearly tight frames on manifolds}

We now review results from our articles \cite{gmcw} and \cite{gmfr}.
Let ${\bf M}$ be a smooth oriented (connected) Riemannian manifold without boundary.

In \cite{gmcw} and \cite{gmfr}, we studied the kernels of operators, from which we
could obtain frames on ${\bf M}$ with frame bounds whose ratio is close to one for
``dilation parmeter'' $a$ close to one. In fact, we started from the
Calder\`on formula
\begin{align}\label{Calderon-Formula}
\int_0^\infty \mid f(ut)\mid^2 dt/t=c<\infty
\end{align}
where $f \in \cal S(\RR^+)$ and $u> 0$. This guarantees that
$f(0)=0$.\\

For $a>1$ and sufficiently close to $1$, by discretizing the
equation (\ref{Calderon-Formula}) one gets, for $u > 0$, the
Daubechies condition

\begin{equation}\label{daub}
0 < A_a \leq \sum_{j=-\infty}^{\infty} |f(a^{2j} u)|^2 \leq B_a < \infty,
\end{equation}

where $B_a/A_a$ converges nearly quadratically to $1$ as
$a\rightarrow 1$.\\
If we let $P$ denote the projection in $L^2({\bf M})$ onto the 
constant functions, which is the null
space of $\Delta$ (the Laplace-Beltrami operator), then by 
applying spectral theory
for $\Delta$ in (\ref{daub}) one obtains the following inequality:

\begin{equation}\label{daub-spec}
0 < A_a (I-P)\leq \sum_{j=-\infty}^{\infty} |f(a^{2j} \Delta)|^2 \leq B_a (I-P)< \infty,
\end{equation}
where $I$ is the identity operator and the sum converges in the strong topology. For $t > 0$, let
$K_t$ be the kernel of the operator $f(t^2 \Delta)$ (which is smooth).
In \cite{gmfr} we verified  that for certain discrete sets $\{x_{j,k}\}$,
where $j$ runs on $\ZZ$ and $1 \leq k \leq N_j$, and for certain positive
weights $\mu_{j,k}$, the system  $\{\phi_{j,k}:=\mu_{j,k}\overline{K}_{a^j}(x_{j,k},\cdot)\}$
 constitutes a ``nearly'' tight frame for $L^2({\bf M})$: There exist
constants (possibly dependent on $a$) $0< A\leq B< \infty$ such that
for any $F \in (I-P)L^2({\bf M})$ one has:
 \begin{align}\label{frame}
 A\parallel F\parallel_2^2 \leq \sum_{j,k} \mid \langle F, \phi_{j,k}\rangle\mid^2 \leq B\parallel F\parallel_2^2;
 \end{align}
further, if the $x_{j,k}$ are chosen appropriately, we can arrange for $B/A$ to be arbitrarily
close to $B_a/A_a$ (which, as we said, converges nearly quadratically to $1$ as
$a\rightarrow 1$).  Thus, for appropriate choices, $B/A$ is arbitrarily close to $1$ 
(which is what we mean when we say that the frame is ``nearly'' tight).  For a precise statement
in the special csse ${\bf M} = S^2$, see Theorem \ref{framain0} below, in the case $s = 0$.

In the course of the proof of Lemma 4.1 of \cite{gmcw} we showed that we could arrange, that the frame element  
$\phi_{j,k}$ is supported in a geodesic ball of radius $a^j$.  Indeed, it suffices to choose 
the function $f$ to have the form $f(\xi^2) = g(\xi)$ for some even $g \in {\cal S}(\RR)$
with \supp$\hat{g} \subseteq (-1,1)$.  Then for some $c$, the Fourier inversion theorem implies
\[ f(t^2 \Delta) = g(t\sqrt{\Delta}) =  c\int_{-1}^{1} \hat{g}(\tau) \cos(\tau t{\sqrt \Delta})F ds. \]
From this and the finite propagation speed property of the wave equation, we see that the
support of $K_t$, the kernel of $f(t^2 \Delta)$, is contained in $\{(x,y) \in {\bf M} \times {\bf M}: 
d(x,y) \leq t\}$, where $d$ denotes geodesic distance.
Since $\phi_{j,k}=\mu_{j,k}\overline{K}_{a^j}(x_{j,k},\cdot)$ for some
constant $\mu_{j,k}$, we see that the support of $\phi_{j,k}$ is indeed contained in a ball of radius
$a^j$ about $x_{j,k}$.

In \cite{gmcw}  we proved that, for any $f$, our wavelets enjoy a property of ``near-diagonal
localization'' which leads to the construction of our frames.  We now explain this property for
the sphere in the spin $s$ situation (the scalar situation of \cite{gmcw} and \cite{narc1}, \cite{narc2} is
the case $s = 0$).  

\subsection{Main results and plan of the paper}

In the spin $s$ situation on the sphere, we have:

\begin{theorem}\label{locestls}
(Near-diagonal localization)
Let $K_t$ be the kernel of $f(t^2\Delta_s)$ where $f \in \cal S(\RR^+)$ and 
$f(0) = 0$.  Then  for every $R, R' \in SO(3)$,
every pair of compact sets ${\cal F}_R \subseteq U_R$ and ${\cal
F}_{R'} \subseteq U_{R'}$, and every pair of $C^{\infty}$
differential operators $X$ (in $x$) on $U_{R'}$ and $Y$ (in $y$) on
$U_{R}$, and for every nonnegative integer $N$, there exists $c$
such that for all $t > 0$, all $x \in {\cal F}_{R'}$ and all $y \in
{\cal F}_{R}$, we have
\begin{equation}
\label{diaglocls0}
\left| XYK_{t,R',R}(x,y) \right|\leq c ~ \frac{ t^{-(2+I+J)}}{(d(x,y)/t)^N},
\end{equation}
where $I:=\deg X$ and $J:=\deg Y$.
\end{theorem}
{\bf Remarks:}\\
1. When $s = 0$ (so that $K_{t,R',R}$ is independent of $R', R$) the near-diagonal
localization theorem is simpler.  In fact, as we showed in
\cite{gmcw}, for some $c$, (\ref{diaglocls0}) holds for all $x,y \in S^2$ and all $t > 0$.
If $f$ has compact support away from the origin, the result was shown earlier for the
sphere in \cite{narc1}, \cite{narc2}.\\
2. Theorem \ref{locestls} was shown in \cite{GeMari08}, in the special case where
$f$ has compact support away from the origin.  It will be shown in general, in 
section \ref{genf} of this article.\\

In section \ref{spinfr} of this article we will use Theorem \ref{locestls} in
order to obtain nearly tight frames of spin wavelets for
$L^2_s(S^2)$. Let us now explain what we mean by spin wavelets, following \cite{GeMari08}:
Associated to the operator $f(t^2\Delta_s)$ there is a kernel $K_t = K_{t,f}$
which can be represented as
\begin{equation}
\label{ktrr}
K_{t,R,R'}(x,y) = \sum_{l \geq |s|} \sum_{\mid m\mid\leq l}  f(t^2\lambda_{ls})\:{}_sY_{lmR}(x)~\overline{{}_sY_{lmR'}}(y).
\end{equation}
For $R \in SO(3)$, and $x \in U_R$, let us define
\begin{equation}
\label{needdf0}
W_{xtR} = \sum_{l \geq |s|} \sum_{\mid m\mid\leq l}  \overline{f}(t^2\lambda_{ls})\:\overline{{}_sY_{lmR}}(x)\:{}_sY_{lm}.
\end{equation}
Then, by definition, $W_{xtR} \in C_s^{\infty}(S^2)$ and the series converges pointwise.
Note that $(W_{xtR})_{R'}(y) = \overline{K}_{t,R,R'}(x,y)$. 

The preceding generalizes the definition we gave in \cite{gmfr} of
wavelets for the case $s=0$. Moreover, if $F \in L^2_s(S^2)$, then
for $x \in U_R$
\begin{equation}
\label{betFdf0}
\beta_{t,F,x,R} := \langle F, W_{xtR}\rangle = (f(t^2\Delta_s)F)_R(x) = (\beta_{t,F})_R(x)
\end{equation}
where we have set $\beta_{t,F}:= f(t^2\Delta_s)F$.  We call
$\beta_{t,F,x,R}$ a {\em spin wavelet coefficient} of $F$.  
%so,
 %by the hypothesis (\ref{Calderon-Formula}) it is obvious that the $s$ spin coherent system $\{W_{xtR}\}_{t,x}$ is admissible for $L^2_s(S^2)$:  
 %\begin{align}
 %\int_{U_R}\int_0^\infty \mid \langle F, W_{xtR}\rangle\mid^2 dxdt/t= c\parallel F\parallel^2\quad \forall ~ F\in L^2_s(S^2).
 %\end{align}

This work aims  to show that, in analogy to the case $s=0$, one can
obtain a nearly tight frame from spin wavelets: 
Let $P=P_{\mid s\mid,s}$ be
the projection onto ${\cal H}_{\mid s\mid, s}$, the null space 
of $\Delta_{s}$ in $C_s^\infty(S^2)$.  Specifically, in section 4 we shall show:

\begin{theorem}
\label{framain0}
Fix $a >1$, and say $c_0 , \delta_0 > 0$.  Suppose $f \in {\mathcal S}(\RR^+)$, 
and $f(0) = 0$.  Suppose that the Daubechies condition $(\ref{daub})$ holds.
Then there exists a constant $C_0 > 0$ $($depending only on $f, a, c_0$ and $\delta_0$$)$
as follows:\\
For $t > 0$, let $K_t = K_{t,f}$ be the kernel of $f(t^2\Delta_s)$,
so that $K_{t,R,R'}$ is as in (\ref{ktrr}). 
Say $0 < b < 1$.
Suppose that, for each $j \in \ZZ$, we can write $S^2$ as a finite disjoint union of measurable sets 
$\{E_{j,k}: 1 \leq k \leq N_j\}$,
where:
\begin{equation} 
 \label{diamleq0}
\mbox{the diameter of each } E_{j,k} \mbox{ is less than or equal to } ba^j, 
\end{equation}
and where:
\begin{equation} 
 \label{measgeq0}
\mbox{for each } j \mbox{ with } ba^j < \delta_0,\: \mu(E_{j,k}) \geq c_0(ba^j)^n.
\end{equation}
$($Such $E_{j,k}$ exist 
provided $c_0$ and $\delta_0$ are sufficiently small, independent of the 
values of $a$ and $b$.$)$
Select $x_{j,k} \in E_{j,k}$ for each $j,k$, and select $R_{j,k}$ with $x_{j,k} \in U_{R_{j,k}}$.

For $1 \leq k \leq N_j$, define $\phi_{j,k} = \mu(E_{j,k})^{1/2}~W_{x_{j,k},a^j,R_{j,k}}$.  Then

\begin{align} 
\label{phijkfr}
(A_a - C_0b)\parallel F\parallel^2 \leq \sum_{j,k}\mid  \langle F,  \phi_{j,k}
\rangle\mid^2 \leq (B_a + C_0b)\parallel F\parallel^2,
\end{align}
for all $F \in (I-P)L^2(S^2)$.  In particular, if $A_a - C_0b > 0$, then
$\left\{\phi_{j,k}\right\}$ is a frame for $(I-P)L^2(S^2)$,
with frame bounds $A_a - C_0b$ and $B_a + C_0b$.
\end{theorem} 
%Then, for $a>1$ sufficiently
%close to $1$ that $A_a > 0$, we can find a discrete set $\{x_{j,k}\} \subseteq S^2$ 
%such that, for certain weights 
%$\mu_{j,k}$ and any selection $\{R_{j,k}\} \subseteq SO(3)$ with
%$x_{j,k} \in U_{R_{j,k}}$, if we set
%$\phi_{j,k}:=  W_{x_{j,k},a^j,R_{j,k}}$, then
%\begin{align} 
%\label{phijkfr}
%A\parallel F\parallel^2 \leq \sum_{j,k} \mu_{j,k} \mid  \langle F,  \phi_{j,k}
%\rangle\mid^2 \leq B \parallel F\parallel^2,
%\end{align}
%for any spin function $F\in (I-P)L_s^2(S^2)$; here $B/A$ is arbitrarily 
%close to $B_a/A_a$.

Note that the choice of $R_{j,k}$ is irrelevant for (\ref{phijkfr}),
since a change in $R_{j,k}$ only multiplies $\phi_{j,k}$ by a factor of absolute 
value $1$.\\

As in the scalar case, we shall also show:
\begin{corollary}
\label{cptspt0}
In Theorem \ref{framain0},  for appropriate $f$ and $C > 0$, one has supp$\phi_{j,k} \subseteq
B(x_{j,k},Ca^j)$.  Thus, for appropriate $a,b$, 
$\left\{\phi_{j,k}\right\}_{j,k}$ is a nearly tight frame for $(I-P)L^2_s(S^2)$,
with supp$(\phi_{j,k}) \subseteq B(x_{j,k},Ca^j)$.
\end{corollary}
Indeed, this will occur, as in the scalar case, if $f$ has the form $f(\xi^2) = g(\xi)$ for some even $g \in {\cal S}(\RR)$
with \supp$\hat{g} \subseteq (-1,1)$. 

The plan of the paper is as follows. 
In section 2, we present a general version of the $T(1)$ theorem for manifolds, which will be an
important tool in the later sections.  In section 3, we use the $T(1)$ theorem to prove a summation 
theorem, which explains when certain kinds of operators on spin functions are bounded on $L^2_s$.
In section 4, we then use the results of section 3 to obtain our main result on frames.  Specifically,
Theorem \ref{framain0} is proved in a more general form in Theorem \ref{framain}.  (In fact, Theorem
\ref{framain} was asserted in \cite{GeMari08}, as Theorem 6.3 of that article; that article referred
to this one for the proof.  See also remark 3 following Theorem 6.3 in \cite{GeMari08}.)
Theorem \ref{locestls} was proved for $f$ with compact support away from $0$ in \cite{GeMari08}; it will 
be proved for general $f$ in section 5. Corollary \ref{cptspt0} will also be proved in section 5, as Corollary 
\ref{cptspt}.

To fully understand this article, it would be helpful if the reader were familiar with sections 1 through 4 of our article
\cite{gmcw}, and up through Theorem 2.4 of \cite{gmfr}.  The facts that we shall need about spin, have almost all
already been explained in this introduction.

\section{A $T(1)$ Theorem for Manifolds}
\label{t1}

As in \cite{gmfr}, we will need an appropriate version of the $T(1)$ 
theorem for manifolds.  We stated a version in Theorem 2.2 of 
\cite{gmfr}.  In this section, we state and prove a version with fewer hypotheses.
We shall work on general 
smooth compact oriented Riemannian manifolds ${\bf M}$.  The notation and
concepts from the introduction will not be needed in this section. 

Let us first review the version of the $T(1)$ theorem
for $\R^n$ in Stein's book \cite{stha}.  (Actually the version below is slightly different
from Stein's; see the appendix of our article \cite{gmst}, with $G = \R^n$ there,
for an explanation of the differences and how they are resolved.)

In $\R^n$, if $r > 0$, we call a $C^1$ function $\omega$ an $r$-bump function if 
it is supported in a ball of radius $r$, if $\|\omega\|_{\infty} \leq 1$,
and if $\|\partial_j \omega\|_{\infty} \leq 1/r$ for each partial derivative
$\partial_j$.  

\begin{theorem}
\label{t1rn}
There exists $C_0 > 0$ as follows.
Suppose $T: C^1_c(\R^n) \to L^2(\R^n)$ is linear, and has a formal adjoint 
$T^*: C^1_c(\R^n) \to L^2(\R^n)$.  Suppose:
\begin{itemize}
\item[$(i)$] $\|T \omega\|_2 \leq A r^{n/2}$ and $\|T^* \omega\|_2 \leq A r^{n/2}$ 
for all $r$-bump functions $\omega$;
\item[$(ii)$] There is a kernel $K(x,y)$, $C^1$ off the diagonal, 
such that if $F \in C_c^{1}$, 
then for $x$ outside the support of $F$, $(TF)(x) = \int K(x,y)F(y)dy$; and 
\item[$(iii)$] Whenever $X$ $($acting in the $x$ variable$)$ and $Y$ (acting in the 
$y$ variable) are in $\{id,\partial_1,\ldots,\partial_n\}$, and at least one of $X$ and $Y$ is the identity map, we have
\[ |XYK(x,y)| \leq A\;|x-y|^{-(n+\deg X + \deg Y)} \]
for all $x, y \in \R^n$ with $x \neq y$;
\end{itemize}
then $T$ extends to a bounded operator on $L^2(\R^n)$, and $\|T\| \leq C_0 A$.
\end{theorem}
We shall need a localized version of this result:
\begin{corollary}
\label{t1rnloc}
Let $B_1, B_2$ be open balls in $\R^n$, of radius $r_1, r_2$ respectively.  Suppose
$K_1 \subseteq B_1$ and $K_2 \subseteq B_2$ are compact.
Then there exists $C_0 > 0$ as follows.

Suppose $T: C^1_c(B_1) \to L^2(B_2)$ is linear, and has a formal adjoint 
$T^*: C^1_c(B_2) \to L^2(B_1)$.   Suppose:
\begin{itemize}
\item[$(a)$] If $f \in C^1_c(B_1)$,
then $Tf = 0$ a.e.\ on $K_2^c$, while $Tf \equiv 0$ if supp$f$ is disjoint from $K_1$.
\item[$(i)$] $\|T \omega\|_2 \leq A r^{n/2}$ for all $r$-bump functions, 
$r \leq r_1$, supported in $B_1$ 
and $\|T^* \omega\|_2 \leq A r^{n/2}$ for all $r$-bump functions,
$r \leq r_2$, supported in $B_2$, and
\item[$(ii)$] There is a kernel $K(x,y)$, $C^1$ off the diagonal, 
such that if $F \in C_c^{1}(B_1)$,
then for $x$ outside the support of $F$, $(TF)(x) = \int K(x,y)F(y)dy$; and 
\item[$(iii)$] Whenever $X$ $($acting in the $x$ variable$)$ and $Y$ (acting in the 
$y$ variable) are in $\{id,\partial_1,\ldots,\partial_n\}$, and at least one of $X$ and $Y$ is the identity map, we have
\[ |XYK(x,y)| \leq A\;|x-y|^{-(n+\deg X + \deg Y)} \]
for all $x \in B_1$, $y \in B_2$ with $x \neq y$;
\end{itemize}
then $T$ extends to a bounded operator on $L^2(B_1)$ into $L^2(B_2)$, and $\|T\| \leq C_0 A$.\\
(b) In (a), we can replace the condition (i) by 
\begin{itemize}
\item[$(i)'$] $\|T \omega\|_{\infty} \leq A$ for all $r$-bump functions, 
$r \leq r_1$, supported in $B_1$, and $\|T^* \omega\|_{\infty} \leq A$,
for all $r$-bump functions, $r \leq r_2$, supported in $B_2$.
\end{itemize}
and the result still follows.
\end{corollary}
{\bf Proof}  Select $\phi, \psi \in C_c^{\infty}(\R^n)$, that supp$\phi \subseteq B_2$,
$\phi \equiv 1$ in a neighborhood of $K_2$, and 
supp$\psi \subseteq B_1$, $\psi \equiv 1$ in a neighborhood of $K_1$.  Define
$\phi T \psi: C^1(\R^n) \to L^2(\R^n)$ by $(\phi T \psi)f = \phi T (\psi f)$. Then 
   $T = \phi T \psi$ on $C_c^1(B_1)$.  For (a),
we need only verify that $\phi T \psi$ satisfies the hypotheses of Theorem \ref{t1rn}.
Evidently $\phi T \psi: C^1(\R^n) \to L^2(\R^n)$ is well-defined, and has
formal adjoint $\overline{\psi}T^*\overline{\phi}$.   Note that for some $C > 0$,
if $\omega$ is an $r$-bump function, then $\psi \omega/C$ is an $r$-bump function 
(if $r \leq r_1$), while $\psi \omega/C$ is an $r_1$-bump function (if $r > r_1$).
It is evident, then, that $\phi T\psi$ satisfies hypothesis (i) of Theorem \ref{t1rn}.
Hypotheses (ii) and (iii) of that Theorem are also satisfied for $\phi T\psi$, whose
kernel is $\phi(x)K(x,y)\psi(y)$.  This proves $(a)$.

For $(b)$, we show that $(i)$ holds (possibly with a different $A$).  
Suppose then that $\omega$ is a bump function for the ball $\{y: |y-x_0| < r\}$,
supported in $B_1$. If $|x-x_0| \geq 2r$, then

\[ \left|(T\omega)(x)\right| \leq \int \left| K(x,y)\omega(y)\right|dy \leq
C \int_{|y-x_0| \leq r} |x-y|^{-n}dy 
\leq Cr^n\max_{|y-x_0| \leq r} |x-y|^{-n}
\leq Cr^n |x-x_0|^{-n}. \]

Thus, by $(i)'$,
\begin{align}\notag
\|T\omega\|_2^2 &= \int_{|x-x_0|<2r} |[ T \omega ](x)|^2 dx
+ \int_{|x-x_0| \geq 2r}  |[ T  \omega ](x)|^2 dx\\\notag
&\leq
C\left[(2r)^n+  r^{2n}\int_{|x-x_0| \geq 2r} |x-x_0|^{-2n}dx\right]\\\notag
&\leq
C[ r^n + r^{2n} r^{-n} ]= Cr^n,
\end{align}
This shows $(i)$ for $T$; similarly for $T^*$.
This completes the proof.\\

Now let ${\bf M}$ be a smooth compact oriented Riemannian manifold of dimension $n$; we will use
Corollary \ref{t1rnloc} to establish a useful version of the $T(1)$ theorem for ${\bf M}$.\\

We shall need the following basic facts, from Section 3 of \cite{gmcw}, about ${\bf M}$ and its
geodesic distance $d$.
For $x \in {\bf M}$, we let $B(x,r)$ denote the ball $\{y: d(x,y) < r\}$.

\begin{proposition}
\label{ujvj}
Cover ${\bf M}$ with a finite collection of open sets $U_i$  $(1 \leq i \leq I)$,
such that the following properties hold for each $i$:
\begin{itemize}
\item[$(i)$] there exists a chart $(V_i,\phi_i)$ with $\overline{U}_i
\subseteq V_i$; and
\item[$(ii)$] $\phi_i(U_i)$ is a ball in $\RR^n$.
\end{itemize}
Choose $\delta > 0$ so that $3\delta$ is a Lebesgue number for the covering
$\{U_i\}$.  Then, there exist $c_1, c_2 > 0$ as follows:\\
For any $x \in {\bf M}$, choose any $U_i \supseteq B(x,3\delta)$.  Then, in
the coordinate system on $U_i$ obtained from $\phi_i$, 
\begin{equation}
\label{rhoeuccmp2}
d(y,z) \leq c_2|y-z|
\end{equation}
for all $y,z \in U_i$; and 
\begin{equation}
\label{rhoeuccmp}
c_1|y-z| \leq d(y,z) 
\end{equation}
for all $y,z \in B(x,\delta)$.
\end{proposition}

We fix collections $\{U_i\}$, $\{V_i\}$, $\{\phi_i\}$ and also $\delta$ as in 
Proposition \ref{ujvj}.  Let $\mu$ be the measure on ${\bf M}$
arising integration with respect to the volume form on ${\bf M}$.
Then: notation as in Proposition \ref{ujvj}, 
there exist $c_3, c_4 > 0$, such that, whenever $x \in {\bf M}$ and $0 < r \leq \delta$,
\begin{equation}
\label{ballsn}
c_3r^n \leq \mu(B(x,r)) \leq c_4r^n
\end{equation}

We fix a finite set
${\mathcal P}$ of real $C^{\infty}$ vector fields on ${\bf M}$, whose
elements span the tangent space at each point.  We let
${\mathcal P}_0 = {\mathcal P} \cup \{\mbox{ the identity map}\}$.\\

Say now $x \in {\bf M}$, and $r > 0$.  We say that a function $\omega \in 
C^1({\bf M})$ is a {\em bump function for the ball} $B(x,r)$, provided $\supp~\omega
\subseteq B(x,r)$, $\|\omega\|_{\infty} \leq 1$ and $\max_{X \in {\mathcal P}} \|X\omega\|_{\infty}
\leq 1/r$.  If  $\omega$ is a bump function for some ball $B(x,r)$, 
we say that $\omega$ is an {\em $r$-bump function}. \\

We can now establish the following $T(1)$ theorem for ${\bf M}$:

\begin{theorem}
\label{t1m}
Let ${\bf M}, d, \mu$ be as above.\\
(a)  Say $r_0 > 0$. Then there exists $C_0 > 0$ as follows.
Suppose $T: C^1({\bf M}) \to L^2({\bf M})$ is linear, and has a formal adjoint 
$T^*: C^1({\bf M}) \to L^2({\bf M})$.   Suppose:
\begin{itemize}
\item[$(i)$] $\|T \omega\|_2 \leq A r^{n/2}$ and $\|T^* \omega\|_2 \leq A r^{n/2}$ 
for all $r$-bump functions $\omega$, if $0 < r \leq r_0$;
\item[$(ii)$] There is a kernel $K(x,y)$, $C^1$ off the diagonal, 
such that if $F \in C^{1}({\bf M})$, 
then for $x$ outside the support of $F$, $(TF)(x) = \int K(x,y)F(y)dy$; and 
\item[$(iii)$] Whenever $X$ $($acting in the $x$ variable$)$ and $Y$ (acting in the 
$y$ variable) are in ${\cal P}_0$, and at least one of $X$ and $Y$ is the identity map, we have
\[ |XYK(x,y)| \leq A\;d(x,y)^{-(n+\deg X + \deg Y)} \]
for all $x, y \in M$ with $x\neq y$ and $d(x,y) \leq 4r_0$;
\item[$(iv)$] Whenever $d(x,y) > 4r_0$, $|K(x,y)| \leq A$;
\end{itemize}
then $T$ extends to a bounded operator on $L^2({\bf M})$, and $\|T\| \leq C_0 A$.\\
(b) In (a), we can replace the condition (i) by 
\begin{itemize}
\item[$(i)'$] $\|T \omega\|_{\infty} \leq A$ and $\|T^* \omega\|_{\infty} \leq A$ 
for all $r$-bump functions $\omega$, if $0 < r \leq r_0$,
\end{itemize}
and the result still follows.
\end{theorem}
{\bf Proof} Without loss of generality, we may assume $4r_0 < \delta$.  Cover ${\bf M}$
by finitely many open balls of radius $r_0$, and select a partition of unity
$\{\zeta_i\}$ subordinate to this covering.  Write
\begin{equation}
\label{tpar}
T = \sum_{i,j} \zeta_i T \zeta_j.
\end{equation}
We prove that each term in the summation in (\ref{tpar}) is bounded on 
$L^2(M)$.  \\

Note first that if supp$\zeta_i \cap$ supp$\zeta_j$ is empty, then those
supports are separated by a positive distance.  In that case,  $\zeta_iT\zeta_j$ is the
operator with the $C^1$ kernel $\zeta_i(x)K(x,y)\zeta_j(y)$.  Thus, in either $(a)$ or $(b)$,
$\zeta_iT\zeta_j$ is bounded on $L^2$, with bound less than or equal to $CA$ for some $C$.\\

Thus we need only need to consider those $i,j$ for which 
supp$\zeta_i$ and supp$\zeta_j$ intersect.  
For some $x_0 \in {\bf M}$,
supp$\zeta_i\cup$supp$\zeta_j \subseteq B(x_0,3r_0)$.  Choose $U_k$ such that
$B(x_0,3\delta) \subseteq U_k$.
Let $B$ be the Euclidean ball $\phi_k(U_k)$; then $\zeta_i T \zeta_j$ may be pulled
over through the diffeomorphism $\phi_k$, to an operator, say ${\cal T}: C_c^1(B) \to L^2(B)$.
Let $K_1$ be the support of $\zeta_j \circ \phi_k^{-1}$, and let $K_2$ be the 
support of $\zeta_i \circ \phi_k^{-1}$.  It is enough to show that
${\cal T}$ satisfies the hypotheses of Corollary \ref{t1rnloc} (with $B=B_1=B_2$); for then we will know that
${\cal T}$ has a bounded extension to $L^2(B)$, and hence that
$\zeta_i T \zeta_j$ has a bounded extension to $L^2({\bf M})$.  Note that hypothesis
$(a)$ of Corollary \ref{t1rnloc} surely holds for $\mathcal T$.\\

Say $B$ has radius 
$s$.  For hypothesis $(i)$ (or in the situation of $(b)$, hypothesis $(i)'$)
of Corollary \ref{t1rnloc}, say $\omega$ is a Euclidean
$r$-bump function supported in $B$, for some $r \leq s$.  We may pull $\omega$ back from $B$ to a function,
$\tilde{\omega} = \omega \circ \phi_k$ on $U_k$.
There, by (\ref{rhoeuccmp2}), for some fixed $c > 0$, $\tilde{\omega}$ is
 a fixed constant  times a  $cr$-bump function on ${\bf M}$, so $\zeta_j \tilde\omega$ is a fixed constant  times a
$cr$-bump function on ${\bf M}$.
We use this if $cr < r_0$.  If instead $cr \geq r_0$, we note that 
$\zeta_j$ is a fixed constant times an $r_0$-bump function on ${\bf M}$, 
while if $X \in {\cal P}$ then $\|X\tilde{\omega}\|_{\infty} \leq C$ independent of $r$, so 
$\zeta_j \tilde{\omega}$ is a fixed constant times an $r_0$-bump function.  From these
observations, we see that, in $(a)$, for some $A' > 0$,
$\|{\cal T}\omega\|_2 \leq A'r^{n/2}$ for all $\omega$ as above.  Moreover in $(b)$
we see that, for some $A' > 0$, $\|{\cal T}\omega\|_{\infty} \leq A'$ for all $\omega$ as above.\\

Similarly we may pull $\zeta_j T^* \zeta_i$ back over to an operator
${\cal R}: C_c^1(B) \to L^2(B)$, and obtain similar conclusions for ${\cal R}$.
To be sure, because of the different metrics, ${\cal R}$ might not equal
${\cal T}^*$.  However, there is a smooth, nowhere vanishing function $h$ on 
$\phi_k(V_k)$, such that whenever $f \in C_c(\phi_k(V_k))$, and if $\tilde{f} = 
f \circ \phi_k \in C_c(V_k)$, then $\int_{\bf M} \tilde{f} d\mu = \int (fh)(x)dx$.
Then a brief computation shows that ${\cal T}$ has formal adjoint
${\cal T}^* = h {\cal R}h^{-1}$.  From this it is now evident that, in $(a)$, ${\cal T}^*$
satisfies hypothesis $(i)$ of Corollary \ref{t1rnloc}, while in $(b)$, it satisfies
hypothesis $(i)'$ of Corollary \ref{t1rnloc}.  

As for hypotheses $(ii)$ and $(iii)$ of Corollary \ref{t1rnloc}, pull over the kernel
of $\zeta_iT\zeta_j$ to $B$,
obtaining the kernel
\[ J(x,y) = \zeta_i \circ \phi_k^{-1}(x) K(\phi_k^{-1}(x),\phi_k^{-1}(y))\zeta_j \circ \phi_k^{-1}(y).\]
The kernel of ${\cal T}$ is evidently $J(x,y)h(y)$.  If $J(x,y) \neq 0$, then because of
the conditions on the supports of $\zeta_i$ and $\zeta_j$, we must have that 
$\phi_k^{-1}(x), \phi_k^{-1}(y) \in B(x_0, 3r_0)$, whose closure is contained in $B(x_0,\delta)$.  We
see then, by (\ref{rhoeuccmp}), that hypotheses $(ii)$ and $(iii)$ of Corollary \ref{t1rnloc}
are satisfied for ${\cal T}$, so the proof is complete.\\

We will have a few other facts that applies to general ${\bf M}$.  At the start of the proof of Theorem
2.3 $(b)$ of \cite{gmfr}, we established $(a)$ of the following proposition.

\begin{proposition}
\label{nxt}
Set
\[ {\mathcal N}_{x,t} = \left\{\Phi \in C^{\infty}({\bf M}):  \;
t^{n} \left|\left(\frac{d(x,y)}{t}\right)^N \Phi(y)\right|\leq 1 
\mbox{ whenever } y \in {\bf M}, \: 0 \leq N \leq n+2\right\}. \]
Say $a > 1$. \\ 
(a)Then there exists $C > 0$ as follows:\\
For each $j \in \ZZ$, write ${\bf M}$ as a finite disjoint union of 
measurable subsets $\{E_{j,k}: 1 \leq k \leq N_j\}$, each of diameter less than $a^j$.
For each $j,k$, select any $x_{j,k} \in E_{j,k}$,\\
Suppose $\{  \Phi_{j,k}\}$ and $\{  \Psi_{j,k}\}$ are two systems of functions such that 
$ \Phi_{j,k},  \Psi_{j,k} \in \mathcal{N}_{x_{j,k}, a^j}$ for all $j,k$, and
suppose $0 \leq J \leq 1$. Then 
\begin{equation}
\label{minnJ}
K_J(x,y):=\sum_{j}\sum_k  a^{-jJ}\mu(E_{j,k})  \left| \Phi_{j,k}(x) \Psi_{j,k}(y)\right|
\leq  C_3d(x,y)^{-n-J}
\end{equation}
for any  $x, y\in \bf{M}$, $x\neq y$.\\
(b) For any $\eta > 0$, the series in (\ref{minnJ}) converges uniformly for $d(x,y) \geq \eta$.
\end{proposition}
{\bf Proof}  As we said, $(a)$ was established at the start of the proof of Theorem
2.3 $(b)$ of \cite{gmfr}.  That proof showed that the convergence was uniform for
$a^{j_0} \leq d(x,y) < a^{j_0+1}$, for any $j_0 \in \Z$.  But under the hypotheses of
$(b)$, $d(x,y)$ is both bounded below and above (by the diameter of ${\bf M}$), so $(b)$
follows.\\

We will also need the following basic facts from Section 3 of \cite{gmcw}: 
\begin{itemize}
\item 
For any $N > n$ there exists $C_N$ such that, for all $x \in {\bf M}$ and $t > 0$,
\begin{equation}
\label{ptestm}
\int_{\bf M} [1 + d(x,y)/t]^{-N} d\mu(y) \leq C_N t^n
\end{equation}
\item
For all $M, t > 0$, and for all
$E \subseteq {\bf M}$ with diameter less than $Mt$, if 
$x_0 \in E$, then one has that
\begin{equation}
\label{alcmpN}
\frac{1}{M+1}[1+d(x,y)/t] \leq [1+d(x_0,y)/t] \leq (M+1)[1+d(x,y)/t]
\end{equation}
for all $x \in E$ and all $y \in {\bf M}$.\\
\end{itemize}
In fact, (\ref{alcmpN}) is a simple consequence of the triangle inequality for $d$.

\section{The Summation Operator}
We return now to the notation in the introduction.
As in the case $s = 0$ studied in \cite{gmfr}, in order
to prove (\ref{phijkfr}), we need to prove the $L^2$-boundedness of 
operators similar to $F \to \sum_{j,k} \mu_{j,k}\langle F,  \phi_{j,k} \rangle \phi_{j,k}$.
(Indeed the boundedness of this operator would clearly at least 
imply the upper bound in (\ref{phijkfr}).)   This is an example of what
we call a {\em summation operator}.\\

To begin, let ${\bf E} = (1,0,0)$ be the ``East pole'' of $S^2$.
Let $P$ denote the projection in $L^2_s(S^2)$ 
onto the null space of $\Delta_s$ in $C_s^{\infty}(S^2)$, $\mathcal H_{s\mid s\mid}$.  This space,
$PL^2_s$, is the span of $2\mid s\mid +1$ elements 
${}_sY_{lm}$ with $l = |s|$.  We shall need the following fact about $PL^2_s$:
\begin{proposition}
\label{pl2sok}
There exists $C > 0$ as follows.  Let $B = B({\bf E},\pi/5)$.
For each $p \in B$, there exists a ${\cal Y}_p \in PL^2_s$ 
such that:\\
(i) ${\cal Y}_{pI}(p) = 1$; \\
(ii) $\|{\cal Y}_{pI}\|_{\infty} \leq C$; and\\
(iii) $\|{\cal Y}_{pI}\|_{C^1(B)} \leq C$.
\end{proposition}
{\bf Proof}  First we claim that for some ${\cal Y} \in PL^2_s$, ${\cal Y}_{I}({\bf E}) = 1$.
If not, then for every $Y \in PL^2_s$, $Y({\bf E}) = 0$ since $Y_R(E)= c Y_I(E)$ for all $R\in SO(3)$.  If $p \in S^2$ is 
arbitrary, then we may select $R \in SO(3)$ with $Rp = {\bf E}$, and then for all
$Y \in PL^2_s$, $(Y^R)_I(p) = Y_R(Rp) = 0$. This implies that $Y^R(p)=0$ for any $Y\in PL^2_s$.   However, the map $Y \to Y^R$ defines an 
action of $SO(3)$ on $PL^2_s$, so the ${}_sY_{lm}^R$, with $l = |s|$, are a basis of
$PL^2_s$.  Thus $Y(p) = 0$ for all $Y \in PL^2_s$ and all $p \in S^2$, contradiction.
So such a ${\cal Y}$ exists. 

Now let us use standard spherical coordinates $(\theta,\phi)$ on the sphere, so that ${\bf E}$ has 
coordinates $(\pi/2,0)$. 
For $|\alpha|,|\beta|,|\phi_0| < \pi/2$, 
consider the unique rotations $S^{\phi_0}_{\alpha}$, $T_{\beta}$ with the properties
\[ T_{\beta}(\theta,\phi) = (\theta,\phi+\beta), \]
\[ S^{\phi_0}_{\alpha}(0,\phi_0) = (\alpha,\phi_0), \mbox{ and } 
S^{\phi_0}_{\alpha} \mbox{ fixes the plane } \phi = \phi_0. \]

If $p = (\theta_0,\phi_0) \in B({\bf E},\pi/4)$, say, we may write $p = R_p{\bf E}$, where
$R_p = S^{\phi_0}_{\theta_0-\pi/2}T_{\phi_0}$.  \\

We claim that $(i)$, $(ii)$ and $(iii)$ are satisfied if we set
\[ {\cal Y}_p = e^{is\psi} {\cal Y}^{R_p^{-1}}, \mbox{ where } \psi = \psi_{pR_pI}. \]
Indeed, $(ii)$ is clear, since ${\cal Y}$ is smooth, and hence bounded.
We have $(i)$ since
\[ {\cal Y}_{pI}(p) = e^{is\psi} ({\cal Y}^{R_p^{-1}})_I(p) = 
({\cal Y}^{R_p^{-1}})_{R_p}(R_p {\bf E}) 
= (({\cal Y}^{R_p^{-1}})^{R_p})_I({\bf E})  = {\cal Y}_{I}({\bf E}) = 1. \]

As for $(iii)$, we note that for any $x \in B$, 
\begin{align}\notag
{\cal Y}_{pI}(x) = e^{is\psi} ({\cal Y}^{R_p^{-1}})_I(x) 
&= e^{is(\psi+\gamma)} ({\cal Y}^{R_p^{-1}})_{R_p}(R_p R_p^{-1}x) \\\notag
&= 
e^{is(\psi+\gamma)}(({\cal Y}^{R_p^{-1}})^{R_p})_I(R_p^{-1}x) \\\notag
&= e^{is(\psi+\gamma)}{\cal Y}_I(R_p^{-1}x) 
\end{align}
where now $\gamma = \psi_{xIR_p}$.  Note that 
$$d(R_p^{-1}x,{\bf E}) = d(x,R_p{\bf E}) = d(x,p) < 2\pi/5,$$
 so 
$R_p^{-1}x \in B({\bf E}, 2\pi/5)$, whose closure is contained in $U_I$.  Also, it is 
evident that $\gamma$ depends smoothly
on both $x$ and $p$.  $(iii)$ now follows at once.
This completes the proof.\\

Next, let
\[ C_{s,0}^{\infty}(S^2) = \left\{\varphi \in C_s^{\infty}(S^2):\; P\varphi = 0\right\}. \]
Cover $S^2$ with finitely many open geodesic balls $\{B(y_m,\pi/10)\}_{m=1}^{M}$, and choose
a partition of unity ${\zeta_m}$ subordinate to that covering.  For each $m$, select
$R_m \in SO(3)$ with $y_m = R_m{\bf E}$.  Since $B({\bf E},\pi/2) \subseteq U_I$,
surely $B(y_m,\pi/2) \subseteq U_{R_m}$.  

As in the previous section, we fix a finite set
${\mathcal P}$ of real $C^{\infty}$ vector fields on $S^2$, whose
elements span the tangent space at each point.  We let
${\mathcal P}_0 = {\mathcal P} \cup \{\mbox{ the identity map}\}$.
We note the following simple fact about ${\mathcal P}$:
\begin{proposition}
\label{rotbdd}
Suppose $0 < \delta < \eta < \pi/2$, $p \in S^2$, and let $Z$ be any smooth vector field
on $B(p,\eta)$.  Then there exists $C > 0$ as follows: for any $C^1$ function $\omega$
on $B(p,\eta)$, and for every $R \in SO(3)$, 
\begin{equation}
\label{rotbddway}
\max_{x \in \overline{B(R^{-1}p,\delta)}}|Z(\omega \circ R)(x)| \leq 
C\max_{X \in {\cal P}}\max_{y \in \overline{B(p,\delta)}}|X\omega(y)|.
\end{equation}
\end{proposition}
{\bf Proof}  We may assume that there are $2$ vector fields $X_1, X_2 \in {\cal P}$,
which are a basis for the tangent space at each point of $\overline{B(p,\delta)}$.  (Otherwise,
we cover $\overline{B(p,\delta)}$ by finitely many closed balls, all contained in $B(p,\eta)$, on 
which this is true, and work on these balls instead of on $\overline{B(p,\delta)}$.)   

The left side of (\ref{rotbddway}) equals $\max_{y \in \overline{B(p,\delta)}}|Z_R\omega(y)|$,
where $Z_R$ is the smooth vector field given by $Z_R g = [Z(g \circ R)] \circ R^{-1}$.  
By working in local coordinates, we see that
we may uniquely write $Z_R(q) = \sum_{k=1}^2 h_k(R,q) X_k(q)$,
for certain continuous functions $h_k$ of $R \in SO(3)$ and $q \in \overline{B(p,\delta)}$.  
From this expression, the proposition is evident.\\

As in the previous section, we can use ${\cal P}_0$ to define $r$-bump functions on $S^2$.
It is evident from Proposition \ref{rotbdd} that for any $\eta < \pi/2$, there is a $C > 0$ such that
whenever $0 < r < \eta$, whenever $R \in SO(3)$, and whenever $\omega$ is an $r$-bump function
supported on Ball $B(p,r)$, then $\omega \circ R/C$ is also an $r$-bump function supported on ball $B(R^{-1}p, \delta)$ for some
$\delta<r$.\\

For any $x \in S^2$, we define
\begin{equation}
\label{msxdef}
\begin{split}
%\left{
{\mathcal M}_{s,x,t} = \left\{\varphi \in C_{s,0}^{\infty}(S^2):  \;
t^{2+\deg Y} \left|\left(\frac{d(x,y)}{t}\right)^N Y\varphi_{R_m}(y)\right|\leq 1, \right. \\ 
%\right} \\
%\left{
\left. \:\forall\: 1 \leq m \leq M,\: y \in \overline{B(y_m,\pi/5)},  
\: 0 \leq N \leq 4 \mbox{ and } Y \in {\mathcal P}_0\right\}.
%\right}
\end{split}
\end{equation}

For example, if $\phi_{j,k}:=  W_{x_{j,k},a^j,R_{j,k}}$ is as in (\ref{phijkfr}),
then by Theorem \ref{locestls}, there is a $C > 0$, such that $\phi_{j,k}/C
\in {\mathcal M}_{s,x_{j,k},a^j}$ for all $j,k$.

Note also that if $\varphi \in {\mathcal M}_{s,x,t}$, then $|\varphi| \in {\mathcal N}_{x,t}$. (as defined
in Proposition \ref{nxt}).
This is evident, once one takes $Y = id$ in the definition of ${\mathcal M}_{s,x,t}$ and
recalls that the $\overline{B(y_m,\pi/5)}$ cover $S^2$.

Note also the following fact:
\begin{itemize}
\item
For every $C_1 > 0$ there exists $C_2 > 0$ such that whenever 
$t > 0$ and $d(x_1,x_2) \leq C_1t$,
we have 
\begin{equation}
\label{nrbyok}
{\mathcal M}_{s,x_1,t} \subseteq C_2{\mathcal M}_{s,x_2,t}.
\end{equation}
\end{itemize}
We can now prove the main facts about the summation operators.  Analogously
to Theorem 2.3 of \cite{gmcw}, we have:
  
\begin{theorem}
\label{sumopthm}
Fix $a>1$.  Then there exists $C_2 > 0$ as follows.

For each $j \in \ZZ$, write $S^2$ as a finite disjoint union of 
measurable subsets $\{E_{j,k}: 1 \leq k \leq N_j\}$, each of diameter less than $a^j$.
(One could consider certain
disjoint subsets of the sphere bounded by longitude and latitude
lines.)
For each $j,k$, select any $x_{j,k} \in E_{j,k}$, and select $\varphi_{j,k}$,
 $\psi_{j,k}$ with \\$\varphi_{j,k} , \psi_{j,k}\in \mathcal{M}_{s,x_{j,k}, a^j}$.
For $F \in C^1_s(S^2)$, we claim that we may define 
\begin{equation}
\label{sumopdf2}
SF = S_{\{\varphi_{j,k}\},\{\psi_{j,k}\}}F = 
\sum_{j}\sum_k \mu(E_{j,k}) \langle  F, \varphi_{j,k}  \rangle  \psi_{j,k}. 
\end{equation}
$($Here, and in similar equations below, the sum in $k$ runs from $k = 1$ to $k = N_j$.$)$
Indeed: 
\begin{itemize}
\item[$(a)$] For any $F \in C^1_s(S^2)$, the series defining $SF$ converges absolutely, uniformly on $\bf{M}$,
\item[$(b)$] $\parallel SF\parallel_2\leq C_2\parallel F\parallel_2$ for all
 $F\in C^1_s(S^2)$.\\
Consequently, $S$ extends to be a bounded operator on $L^2_s(S^2)$, with norm less than
or equal to $C_2$. 
\item[$(c)$] If $F \in L^2_s({S^2})$, then
\begin{equation}\notag%\label{uncndl2}
SF = \sum_{j}\sum_k \mu(E_{j,k}) \langle  F, \varphi_{j,k}  \rangle  \psi_{j,k}
\end{equation}
where the series converges unconditionally. 
\item[$(d)$] If $F,G \in L^2_s({S^2})$, then 
\begin{equation}\notag%\label{eq:wk}
 \langle  SF,G  \rangle = 
\sum_{j}\sum_{k}\mu(E_{j,k}) \langle  F, \varphi_{j,k} \rangle   \langle \psi_{j,k},G \rangle ,
\end{equation}
where the series converges absolutely.\\
\end{itemize}
\end{theorem}
{\bf Proof}  We first show $(a)$ and $(b)$.  As we shall explain, $(c)$ and $(d)$ are
immediate consequences.

For $(a)$ and $(b)$, it suffices to show that for each $m,q$ with $1 \leq m,q \leq M$, we may define
\begin{equation}
\label{sumopdfpq2}
S_{mq}F =
\sum_{j}\sum_k \mu(E_{j,k}) \langle  \zeta_q F, \varphi_{j,k}  \rangle  \zeta_m\psi_{j,k}
\end{equation}
initially for $F \in C_s^1(S^2)$, with $(a)$ and $(b)$ satisfied for each $S_{mq}$ in place
of $S$.  For instance, $(a)$ for $S$ will follow, since the termwise sum of a finite number
of absolutely convergent series is surely absolutely convergent.  Of course, once these results 
are proved, we will know that each $S_{mq} = \zeta_m S \zeta_q$.  

%Say first that if supp$\zeta_m \cap$ supp$\zeta_q$ is empty; then those
%supports are separated by a positive distance, say $\eta$.  In that case, even if we only
%know that $F \in L^1_s(S^2)$, we have for the $(j,k)$ term in (\ref{sumopdfpq2}),
%\begin{equation}
%\label{l1sep}
%|\mu(E_{j,k}) \langle  \zeta_q F, \varphi_{j,k}  \rangle  \psi_{j,k}(x)|
%\leq \int |F(y)| |\zeta_q(y)\varphi_{j,k}(y)| dS(y) |\zeta_m(x) \psi_{j,k}(x)|.
%\end{equation}

%We have that $|\zeta_q\varphi_{j,k}|, |\zeta_m \psi_{j,k}| \in {\mathcal N}_{x_{j,k}, a^j}$.
%Thus, by Proposition \ref{nxt} (b), the sum
%\begin{equation}
%\label{kerdsj}
%\sum_{j,k} |\zeta_q(y)\varphi_{j,k}(y)| |\zeta_m(x) \psi_{j,k}(x)|
%\end{equation}
%converges uniformly on $S^2 \times S^2$.  Thus $(a)$ for $S_{mq}$ is immediate, even
%for $F \in L^1_s(S^2)$.  Moreover we see as well that $S_{mq}$ as defined in 
%(\ref{sumopdfpq2}) is a bounded operator from $L^1_s(S^2)$ to $L^{\infty}_s(S^2)$,
%from which $(b)$ for $S_{mq}$ is immediate as well.

%Say now instead that supp$\zeta_m \cap$ supp$\zeta_q$ is not empty. 
Recall that supp$\zeta_q \subseteq B(y_q,\pi/10)$.  For $(a)$ for $S_{mq}$, then, it is
enough to show that for any $F \in C^1_s$, with $\supp F \subseteq B(y_q,\pi/10)$,
$\sum_{j}\sum_k \mu(E_{j,k}) \langle  F, \varphi_{j,k}  \rangle  \psi_{j,k}$ converges
absolutely, uniformly on $S^2$.  (As usual, supp$F$ is the closure of $\{x: F(x) \neq 0\}$.
Also, we won't need to localize with $\zeta_m$ 
in this part of the proof.)  In fact, since for any such $F$, $F_{R_q}$ is a multiple
of an $r$-bump function for $r = \pi/10$, it is enough to show the following:\\
\ \\
(*) Say $F \in C^1_s$, with supp$F \subseteq B(y_q,\pi/10)$, and that $F_{R_q}$
is an $r$-bump function for some $r \leq \pi/10$.  Then for all $x \in S^2$ and $R\in SO(3)$,
\[\sum_{j}\sum_k \mu(E_{j,k}) |\langle  F, \varphi_{j,k}  \rangle| |\psi_{j,k}(x)| \leq C_1. \]
(Here and in the sequel, for a spin function $\psi$,
$|\psi(x)| := |(\psi_R(x)|$ for any $R$ with $x \in U_R$.)
Here the convergence is uniform on $S^2$, and the constant $C_1$ may be chosen independently
of the choice of $\{\varphi_{j,k}\}, \{\psi_{j,k}\}$, $E_{j,k}$, $x_{j,k}$, $r$ or $F$.\\
\ \\
To prove (*), first, for each $j$, let
\[ B_j = \sup_k \left|  \langle  F, \varphi_{j,k} \rangle \right|. \]
 
Since $ \left| \langle  F, \varphi_{j,k} \rangle \right| \leq \|F\|_1\|\varphi_{j,k}\|_{\infty}$
we surely have that 

\begin{equation}
\label{cjrest1}
B_j \leq Cr^2 a^{-2j}.
\end{equation}
(Here, and in the rest of the proof of (*), $C$ denotes a constant, which may change from one
appearance to the next, but which may be chosen independently of the choice of
$\{\varphi_{j,k}\}, \{\psi_{j,k}\}$, $r$ or $F$.)\\

On the other hand, we claim that if $a^j \leq r$, then
\begin{equation}
\label{cjrest2}
B_j \leq Ca^j/r.
\end{equation}

To see this, we first show that there exists $C_0 > 0$ as follows.  Let
  $R \in SO(3)$ and
$\omega$ be  an $r$-bump function, $r < \pi/10$, with supp$\omega \subseteq B(R{\bf E},\pi/10)$.
Then for all $x,u \in S^2$,
\begin{equation}
\label{bmpest}
|\omega(x) - {\cal Y}_{pI}(R^{-1}x)\omega(u)| \leq C_0d(x,u)/r,
\end{equation}
where  ${\cal Y}_p$ is as in Proposition \ref{pl2sok}, for $p = R^{-1}u$.
(To be clear, ${\cal Y}_{pI}(R^{-1}x)\omega(u)$ is taken to be zero here if $\omega(u) = 0$,
and in particular if $u \notin B(R{\bf E},\pi/10)$, 
even if ${\cal Y}_{pI}(R^{-1}x)$ is not defined.)  To prove (\ref{bmpest}), note that 
the left side is surely bounded (independent of choice of $R, r, \omega$, $x$ or $u$),
by Proposition \ref{pl2sok} $(ii)$ -- since, in checking this, we may assume that $u \in 
B(R{\bf E},\pi/10)$, in which case $p \in B({\bf E}, \pi/10)$.  Since the left side is bounded, we 
may also assume that $d(x,u) \leq r < \pi/10$.  The left side of (\ref{bmpest}) is
zero unless at least one of $x,u \in B(R{\bf E},\pi/10)$, and hence we may assume that both
$x,u \in B(R{\bf E},\pi/5)$.  Since   $\omega \circ R_1 $ is a fixed constant times an $r$-bump
function for all $R_1 \in SO(3)$, by writing $\omega(x) = (\omega \circ R)(x')$,
$\omega(u) = (\omega \circ R)(u')$, $x'=R^{-1}x$, $u'=R^{-1}u=p$ in (\ref{bmpest}), we see that in (\ref{bmpest}), we
may take $R = I$, in which case $u=p$.  But, if we note that $\omega(p)-{\cal Y}_{pI}(p)\omega(p) = 0$,
we now see that (\ref{bmpest}) (for $R=I$, $x,p \in B({\bf E},\pi/5)$) follows at once from
Proposition \ref{pl2sok} $(iii)$ and the mean value theorem.  (One can use stereographic 
projection onto the tangent plane to $S^2$ at ${\bf E}$ before applying the mean value theorem.) 
This proves (\ref{bmpest}) in full generality.

We also note that in (\ref{bmpest}), we have 
\begin{equation}
\label{ypir}
{\cal Y}_{pI}(R^{-1}x) =
(({\cal Y}_p^{R^{-1}})^{R})_I(R^{-1}x) = ({\cal Y}_p^{R^{-1}})_{R}(x).
\end{equation}
(Here we are abbreviating ${\cal Y}_p^{R^{-1}} = ({\cal Y}_p)^{R^{-1}}$.)

To prove (\ref{cjrest2}) we apply (\ref{bmpest}) with $R=R_q$, $\omega = F_{R_q}$, $u=x_{j,k}$.  
Since each $\varphi_{j,k}$ is in $C_{s,0}^{\infty}$, it is orthogonal to $PL^2_s$.
Thus for any $j,k$, if $p = R_q^{-1}x_{j,k}$, then
\[ \langle F, \varphi_{j,k} \rangle = \langle F - {\cal Y}_p^{R_q^{-1}}, \varphi_{j,k} \rangle 
= \langle F_{R_q} - ({\cal Y}_p^{R_q^{-1}})_{R_q}, (\varphi_{j,k})_{R_q} \rangle.\]
Thus, by (\ref{bmpest}) and (\ref{ypir}),
 \begin{align}\label{goholder}\left| \langle  F,  \varphi_{j,k} \rangle \right|
 &\leq C\int r^{-1}d(x,x_{j,k}) \left| \varphi_{j,k}(x)\right| d\mu(x)\\\notag
&\leq Cr^{-1} a^{-2j}a^j\int \left[1+d(x,x_{j,k})/a^j\right]^{-3}d\mu(x)\\\label{endholder}
&\leq Ca^j/r 
  \end{align}
by (\ref{ptestm}).  Now, for any $j$ and any $y \in S^2$, we have
\begin{align}
\sum_k \mu(E_{j,k}) \left|  \langle  F,  \varphi_{j,k} \rangle \right|\left|  \psi_{j,k}(y)\right|&\leq C_{j,r}\sum_k
\mu(E_{j,k}) \left| \psi_{j,k}(y)\right|
\label{sumsetup}\\
&\leq C C_{j,r} a^{-jn}\sum_k \mu(E_{j,k}) \left[1+d(y,x_{j,k})/a^j\right]^{-3}
\label{sumtogo}
\\
&\leq C  C_{j,r} a^{-jn}\int_{S^2}\left[1+d(y,x)/a^j\right]^{-3} d\mu(x) \label{sumgone}\\
&\leq   C  C_{j,r},  \label{ccjrhere}
\end{align}
again by (\ref{ptestm}).  (In passing from (\ref{sumtogo}) to (\ref{sumgone}), we used
(\ref{alcmpN}).)
 
Taking  the sum over $j$ we have:
\begin{align}\notag%\label{sum-ov-j}
\sum_j\sum_k   \mu(E_{j,k}) \left|  \langle  F,  \varphi_{j,k} \rangle \right| \left|  
(\psi_{j,k})_R(y)\right|&\leq C   \sum_j   C_{j,r}  \\\notag
& \leq C\left[\sum_{a^j \leq r}  r^{-1}a^j + \sum_{a^j > r}   r^{n}a^{-jn}\right] \\
& \leq C_1, \label{c1here}
\end{align}
by (\ref{cjrest1}) and (\ref{cjrest2}).  This proves (*).  Thus $(a)$ is established
for $S_{mq}$, and hence for $S$.  

%To obtain $(b)$ for $S$, we must still prove $(b)$ for $S_{mq}$ when
%supp$\zeta_m \cap$ supp$\zeta_q$ is not empty.  In that case,
%supp$\zeta_m \cup$ supp$\zeta_q \subseteq B(y_q,3\pi/10)$, whose closure
%is contained in $B(y_q,\pi/2)$, which is contained in $U_{R_q}$.  
Next, we will prove $(b)$ for 
$S_{mq}$, by applying Theorem \ref{t1m} to a scalar operator which is a ``trivialization'' of $S_{mq}$.

If $\zeta, \zeta' \in C^{\infty}(S^2)$, let us write $\zeta \prec \zeta'$ if $0 \leq \zeta, \zeta' \leq 1$
and if $\zeta' = 1$ in a neighborhood of supp$\zeta$.  Choose, then, $\zeta, \zeta'$ supported
in $B(y_q,\pi/10)$, with $\zeta_q \prec \zeta \prec \zeta'$.  

Evidently the map $F \to F_{R_q}$ is a unitary map from $L^2_s$ onto $L^2$; let ${\cal T}$ denote the
inverse of this map.  Note that if $\omega \in C^1(S^2)$ has support contained in $U_{R_q}$,
then ${\cal T}(\omega) \in C^1_s$ (since it vanishes in neighborhoods of 
$R_q{\bf N}$ and $R_q{\bf S}$).  Consequently, by $(a)$, we may define an operator 
$\tilde{S}_{mq}: C^1 \to L^{\infty}$ by
\[ \tilde{S}_{mq}\omega = [S_{mq} {\cal T}(\zeta \omega)]_{R_m}. \]
For $(b)$ for $S_{mq}$, it is enough to show that, for some $C > 0$, 
independent
of the choice of $\{\varphi_{j,k}\}, \{\psi_{j,k}\}$, $E_{j,k}$, or $x_{j,k}$, we have\\
\ \\
(**)\ \ \ \ \ \ \ \ \ \ \ \ \ \ \ \ \ \ \ \ \ \ \ \ \ \ \ \ \ \ \ \ \ \ \ \ \ \ \ \ \ \ \ \ \ \ \ \ \ \ %
$\|\tilde{S}_{mq}\omega\|_2 \leq C\|\omega\|_2$,\\
\ \\
for all $\omega \in C^1$.  For if (**) were known, then for all $F \in C_1^s$ we would have
that for all $F \in C^1_s$,
\[ \|S_{mq}F\|_2 =  \|S_{mq}(\zeta \zeta' F)\|_2 =
\|\tilde{S}_{mq}(\zeta' F_{R_q})\|_2 \leq C\|\zeta' F_{R_q}\|_2 \leq C\|F\|_2, \]
which will prove $(b)$.  (In the rest of the proof of $(b)$, $C$ will always denote a positive
constant, which may change from one appearance to the next, but which is 
independent
of the choice of $\{\varphi_{j,k}\}, \{\psi_{j,k}\}$, $E_{j,k}$, or $x_{j,k}$.)
\ \\
Let us then verify that the hypotheses of Theorem \ref{t1m} hold for $\tilde{S}_{mq}$,
where we use $(i)'$ of that theorem in place of $(i)$, and where we set
$r_0$ in that theorem to be $\pi/10$.  Firstly, let $\omega$ be an $r$-bump
function, $r \leq \pi/10$.  Then $\zeta_q \omega$ is a fixed constant times an $r$-bump function.
Let $F = \zeta_q {\cal T}(\zeta \omega)$; then $F_{R_q} = \zeta_q \omega$ is a fixed constant times an $r$-bump function,
and supp$F \subseteq B(y_q,\pi/10)$.  
We then have, by $(*)$, that $\|\tilde{S}_{mq}(\omega)\|_{\infty} \leq \|S(F)\|_\infty \leq C$
for some $C$, which gives us part of $(i)'$.  
To complete the proof of $(i)'$, we also need to consider $\tilde{S}_{mq}^*$.  
Note that for any $\omega \in C^1$ we have
\begin{equation}
\label{smqxp}
\tilde{S}_{mq}\omega = \sum \mu(E_{j,k})\langle \omega, \zeta_q (\varphi_{j,k})_{R_q} \rangle \zeta_m (\psi_{j,k})_{R_m}
\end{equation}
where the series converges absolutely.   Thus, for any $\omega, \nu \in C^1$, we have
\[ \langle \tilde{S}_{mq}\omega, \nu \rangle
 = \sum_j \sum_k \mu(E_{j,k})\langle \omega, \zeta_q (\varphi_{j,k})_{R_q} \rangle 
\langle \zeta_m (\psi_{j,k})_{R_m}, \nu \rangle, \]
where the sum converges absolutely.  
Consequently, the formal adjoint of $\tilde{S}_{mq}$ is $\tilde{S}^*_{mq}$, where, if $\nu \in C^1_s$,
\[ \tilde{S}^*_{mq}\nu = \sum_{j}\sum_k \mu(E_{j,k}) \langle  \nu, \zeta_m (\psi_{j,k})_{R_m}
  \rangle \zeta_q (\varphi_{j,k})_{R_q}. \]
Since this is of the same form as $\tilde{S}_{mq}$, we obtain $(i)'$ for $\tilde{S}_{mq}$, in its entirety.\\

As for $(ii)$, note that if the support of $\omega$ is contained in a compact set $T \subseteq S^2$,
and $x \notin T$, then there is a $\eta > 0$ such that $d(x,y) \geq \eta$ for all $y \in T$.  
Thus, by (\ref{smqxp}), $\tilde{S}_{mq}$ has 
kernel $K$ (in the sense of Theorem \ref{t1m} $(b)$), where
\begin{equation}
\label{kerint}
K(x,y) = \sum_{j,k} \zeta_q(y)(\varphi_{j,k})_{R_q}(y)\zeta_m(x) (\psi_{j,k})_{R_m}(x)
\end{equation}
since, by Proposition \ref{nxt} (b), the series converges absolutely and uniformly
for $d(x,y) \geq \eta$.
%   (We made a similar observation when discussing (\ref{kerdsj}).)

By definition of ${\mathcal M}_{s,x,t}$, if $\varphi \in {\mathcal M}_{s,x,t}$, 
then for some $C > 0$, 
then $t^{\deg Y} Y(\zeta_{\alpha}\varphi_{R_\alpha})/C \in {\mathcal N}_{x,t}$ for every $Y \in {\mathcal P}_0$
and every $\alpha$ ($1 \leq \alpha \leq M$).  We apply this with $t = a^j$ ($j \in \ZZ$).
It therefore follows from Proposition \ref{nxt} that, whenever $X$ (acting in the $x$ variable) and $Y$ (acting in the 
$y$ variable) are in ${\mathcal P}_0$, then 
\begin{equation}
\label{drvadup}
\sum_{j}\sum_k  \mu(E_{j,k})  \left| X[\zeta_m (\psi_{j,k})_{R_m}](x)
Y[\zeta_q(\overline{\varphi}_{j,k})_{R_q}](y) \right|
\leq  C_3d(x,y)^{-(2+\deg X + \deg Y)}.
\end{equation}
provided $\deg X + \deg Y \leq 1$.

Working in local
coordinates on $U_{R_m}$ we could take $X$ in (\ref{drvadup}) to be any $\partial/\partial x_l$,
and working in local coordinates on $U_{R_q}$, we could take 
$Y$ to be any $\partial/\partial y_l$. We thus see that $K(x,y)$ is $C^1$ off the 
diagonal.  Finally we see that we may bring derivatives
past the summation sign in (\ref{kerint}), and thus, for any $X, Y \in {\mathcal P}_0$,
\[ \left|XYK(x,y)\right| \leq C_3 d(x,y)^{-(2+\deg X + \deg Y)} \]
for $x \neq y$, provided $\deg X + \deg Y \leq 1$. 
This shows that $\tilde{S}_{mq}$ satisfies conditions $(ii)$, $(iii)$ and $(iv)$ of Theorem \ref{t1m},
which completes the proof of $(b)$.
\begin{proof}[Proof of (c)]
It is evident, by (a), that (c) holds for $F \in C^1_s$.

Suppose now that $\mathcal{F} \subseteq \{(j,k): j \in \ZZ, 1 \leq k \leq N_j\}$ is finite, and define
$S^{\mathcal{F}}: L^2_s \rightarrow C^1_s$ by 
\begin{align}\notag
S^{\mathcal{F}}F=
\sum_{(j,k) \in {\cal F}} \mu(E_{j,k}) \langle  F, \varphi_{j,k}  \rangle  \psi_{j,k}
\end{align}
By (b), $S^{\mathcal{F}}: L^2_s \rightarrow L^2_s$ is bounded, with norm
$\|S^{\mathcal{F}}\| \leq C_2$
{\em for all } ${\mathcal F}$.  (Indeed, in the formula for $S$,
we have just replaced the $\varphi_{j,k}$ by $0$
if $j,k \notin {\cal F}$, and $0$ is surely in
all ${\mathcal M}_{s,x_{j,k},a^j}$.)  Since (c) is true for $F \in C^1_s$, which
is dense in $L^2_s$, it now follows for all $F \in L^2_s$.
\end{proof} 
\begin{proof}[Proof of (d)]
 
This follows at once from (c), since a series of complex numbers conveges 
unconditonally if and only if it converges absolutely.\end{proof}

\section{Frames}\label{spinfr}
As we shall now explain, it is not difficult to now adapt the arguments of \cite{gmfr}
to obtain nearly tight frames of spin wavelets for $S^2$.  Besides accounting for spin,
we make another change in the arguments of \cite{gmfr}:  in \cite{gmfr}, we looked
at certain charts on the manifold; on $S^2$, it is natural to use the charts obtained through
stereographic projections.

Recall that in section 2 of \cite{gmcw}, we proved the following general fact.

\begin{lemma}
\label{gelmaylem}
Let $T$ be a positive self-adjoint operator on a Hilbert space ${\cal H}$, and
let $P$ be the projection onto the null space of $T$.  
Suppose $l \geq 1$ is an integer, $f_0 \in {\cal S}(\RR^+)$, 
$f_0 \not\equiv 0$, and let $f(s) = s^l f_0(s)$.  
Suppose that $a > 0$  $(a \neq 1)$
is such that the {\em Daubechies condition} holds: for any $u > 0$, 
\begin{equation}
\label{daubrep}
0 < A_a \leq \sum_{j=-\infty}^{\infty} |f(a^{2j} u)|^2 \leq B_a < \infty,
\end{equation}
Then 
$\lim_{M, N \rightarrow \infty} 
\left[\sum_{j=-M}^N |f|^2(a^{2j} T)\right]$
exists in the strong operator topology on ${\cal H}$; we denote this limit by 
$\sum_{j=-\infty}^{\infty} |f|^2(a^{2j} T)$.  Moreover
\begin{equation}
\label{strsum1}
A_a(I-P) \leq \sum_{j=-\infty}^{\infty} |f|^2(a^{2j} T) \leq
B_a(I-P). 
\end{equation}
\end{lemma}

We use this general result to obtain nearly tight frames of spin wavelets:
\begin{theorem}
\label{framain}
$(a)$ Fix $a >1$, and say $c_0 , \delta_0 > 0$.  Suppose $f \in {\mathcal S}(\RR^+)$, 
and $f(0) = 0$.
Then there exists a constant $C_0 > 0$ $($depending only on $f, a, c_0$ and $\delta_0$$)$
as follows:\\
Let ${\mathcal J} \subseteq \ZZ$ be cofinite (that is, assume ${\mathcal J}^c$ is finite).
(The main interest is in the case ${\mathcal J} = \ZZ$.$)$
For $t > 0$, let $K_t = K_{t,f}$ be the kernel of $f(t^2\Delta_s)$,
so that $K_{t,R,R'}$ is as in (\ref{ktrr}).  
Say $0 < b < 1$.
Suppose that, for each $j \in {\mathcal J}$, we can write $S^2$ as a finite disjoint union of measurable sets 
$\{E_{j,k}: 1 \leq k \leq N_j\}$,
where:
\begin{equation}
\label{diamleq}
\mbox{the diameter of each } E_{j,k} \mbox{ is less than or equal to } ba^j, 
\end{equation}
and where:
\begin{equation}
\label{measgeq}
\mbox{for each } j \mbox{ with } ba^j < \delta_0,\: \mu(E_{j,k}) \geq c_0(ba^j)^n.
\end{equation}
$($By the remarks just before Theorem 2.4 of \cite{gmfr}, there are $c_0', \delta > 0$ such that we
can this is always possible if $c_0 \leq c_0^{\prime}$ and $\delta_0 \leq 2\delta$.  In fact, it is easily
seen that one can find such $E_{j,k}$ bounded by latitude and longitude lines.$)$
Select $x_{j,k} \in E_{j,k}$ for each $j,k$, and select $R_{j,k}$ with $x_{j,k} \in U_{R_{j,k}}$.
Let
\begin{equation}
\label{phwt}
\Phi_{x_{j,k},a^j} = W_{x_{j,k},a^j,R_{j,k}}
\end{equation}
  be as in (\ref{needdf0}).  
By Theorem 1.1, there is a constant $C$ $($independent of the
choice of $b$ or the $E_{j,k}$$)$, such that $\Phi_{x_{j,k},a^j}\in C{\mathcal M}_{s,x_{j,k},a^j}$
for all $j,k$.  Thus, if for $1 \leq k \leq N_j$, we set
\[ \varphi_{j,k} = 
\begin{cases}
\Phi_{x_{j,k},a^j} & \mbox{ if}\;\; j \in {\mathcal J},  \cr 
0  & \text{otherwise},  \cr
\end{cases}
\]
we may thus form the summation operator $S^{\mathcal J}$ with

\begin{equation}\notag%\label{sumjopdf2}
S^{\mathcal J}F = S_{\{\varphi_{j,k}\},\{\varphi_{j,k}\}}F = 
\sum_{j}\sum_k \mu(E_{j,k}) \langle  F, \varphi_{j,k}  \rangle  \varphi_{j,k}. 
\end{equation}
 and Theorem $\ref{sumopthm}$ applies.\\
 
Let $Q^{\mathcal J} = \sum_{j \in {\mathcal J}} |f|^2(a^{2j} \Delta_s)$ 
(strong limit, as guaranteed by Lemma \ref{gelmaylem}). 
  Then for all $F \in L^2_s(S^2)$,

\begin{equation}
\label{qsclose}
\left| \langle  (Q^{\mathcal J}-S^{\mathcal J})F,F  \rangle \right| \leq C_0b  \langle  F,F  \rangle 
\end{equation}
$($or, equivalently, since $Q^{\mathcal J}-S^{\mathcal J}$ is self-adjoint, 
$\|Q^{\mathcal J}-S^{\mathcal J}\| \leq C_0b$.$)$\\  
$(b)$ In $(a)$, take ${\mathcal J} = \ZZ$; set $Q = 
Q^{\ZZ}$, $S = S^{\ZZ}$.  Use the notation of Lemma \ref{gelmaylem}, where now 
$T = \Delta_s$ on $L^2_s({S^2})$;
suppose in particular that the Daubechies condition $(\ref{daubrep})$ holds.  
Then,
if $P$ denotes the projection in $L^2_s$ onto ${\cal H}_{|s|,s}$, the null space of $\Delta_s$, we have
\begin{equation}
\label{snrtgt}
(A_a-C_0b)(I-P) \leq S \leq (B_a + C_0b)(I-P)
\end{equation}
as operators on $L^2_s(S^2)$.  Thus, for any $F \in (I-P)L^2_s(S^2)$,
\begin{equation}\notag%\label{snrtgt1}
(A_a-C_0b)\|F\|^2 \leq \sum_{j,k} \mu(E_{j,k})| \langle  F,\varphi_{j,k}  \rangle |^2
\leq (B_a + C_0b)\|F\|^2,
\end{equation}
so that, if $A_a - C_0b > 0$, then
$\left\{ \mu(E_{j,k})^{1/2}\varphi_{j,k}\right\}_{j,k}$ is a frame for $(I-P)L^2_s(S^2)$,
with frame bounds $A_a - C_0b$ and $B_a + C_0b$.
\end{theorem}  
{\bf Note}
By Lemma 7.6 of \cite{gmst}, $B_a/A_a = 1 + O(|(a-1)^2 (\log|a-1|)|)$; thus, evidently,
$(B_a + C_0b)/(A_a - C_0b)$ can be made arbitrarily close to $B_a/A_a$ by
choosing $b$ sufficiently small.  So we have constructed ``nearly tight" frames
for $(I-P)L^2_s({S^2})$.
\begin{proof}
  We prove $(a)$.  To simplify the notation, in the proof of (a), $j$ will
always implicitly be restricted to lie in ${\mathcal J}$, unless otherwise explicitly stated.

Since $\delta_0$ only occurs in (\ref{measgeq}), we may assume $\delta_0 < \pi/10$; for otherwise, we may
replace $\delta_0$ by any positive number less than $\pi/10$, and (\ref{measgeq}) still holds.  

The definition of $S^{\mathcal J}$ evidently does not depend on the choice of $R_{j,k}$.  
For the proof of $(a)$, let us fix, for each $j,k$, an $m = m(j,k)$ with
$x_{j,k} \in B(y_m,\pi/10)$.  Then $R_m^{-1}x_{j,k} \in
B(R_m^{-1}y_m,\pi/10) = B({\bf E},\pi/10) \subseteq U_I$,
so surely $x_{j,k} \in U_{R_m}$.  Let us take $R_{j,k} = R_{m(j,k)}$,
so that we do have $x_{j,k} \in U_{R_{j,k}}$.  

Since $Q^{\mathcal J}$ and $S^{\mathcal J}$ are bounded operators, we need only show (\ref{qsclose})
for the dense subspace $C^1_s$.  For $b > 0$, we let $\Omega_b = 
\log_a (\delta_0/b)$, so that $j < \Omega_b$ is equivalent to 
$ba^j < \delta_0$.  

Note that, if $F \in L^2_s$, $x \in S^2$ and $x \in U_R$, 
then $|[f(t^2\Delta_s)F](x)| = |\langle W_{x,t,R},F \rangle|$, independent of 
choice of $R$.  Let us denote this quantity by $|\langle W_{x,t},F \rangle|$,

Observe now that 
\begin{align}\notag
\left| \langle  (Q^{\mathcal J}-S^{\mathcal J})F,F  \rangle \right|
& = \left|\sum_j \int_{S^2} \left| \langle  W_{x,a^j},F  \rangle  \right|^2 d\mu(x)
- \sum_j \sum_k \mu(E_{j,k})\left| \langle  W_{x_{j,k},a^j},F  \rangle  \right|^2\right|\\\notag
& \leq I + II,
\end{align}
where
\[ I = \left| \sum_{j < \Omega_b}\left[\int_{S^2} \left| \langle  W_{x,a^j},F  \rangle  \right|^2 d\mu(x) -
\sum_k \mu(E_{j,k})\left| \langle  W_{x_{j,k},a^j},F  \rangle  \right|^2\right] \right|, \]
and
\[ II = \sum_{j \geq \Omega_b} \int_{S^2} | \langle  W_{x,a^j},F  \rangle  |^2 d\mu(x)
+ \sum_{j \geq \Omega_b} \sum_k \mu(E_{j,k})| \langle  W_{x_{j,k},a^j},F  \rangle  |^2. \]

For $II$, we need only note that, by Theorem 4.2(a) of \cite{GeMari08}, there exist
$C,M > 0$ such that $\|{}_sY_{lm}\|_{\infty} \leq C(l+1)^M$ for all $l,m$.  
Select $Q > 2M+2$.  Then, by (\ref{needdf0}), for any $x,t,R$, if $x \in U_R$ and $t > 1$, then
\[ \|W_{x,t,R}\|_2^2 \leq C\sum_{l=|s|+1}^{\infty} |f(t^2\lambda_{ls})|^2 (l+1)^{2M+1} 
\leq C t^{-2Q} \sum_{l=1}^{\infty} (l+1)^{2M-2Q+1} < Ct^{-2Q} \]
since $f \in {\cal S}(\RR^+)$, and since, in the first summation, we always have
$t^2\lambda_{ls} \geq |s|+2$.  

In particular, there exists $C > 0$ such that, for any $x,R$ with $x \in U_R$, and any $t > 1$,
$\|W_{x,t,R}\|_2^2 \leq C/t$.  Accordingly,

\[ II\leq C \sum_{j \geq \Omega_b} a^{-j}\mu(S^2)\|F\|_2^2 \leq Ca^{-\Omega_b}\|F\|_2^2
= Cb\|F\|_2^2. \]
\ \\
Thus we may focus our attention on $I$.  We have 
\begin{align}\notag
I & = \left| \sum_{j < \Omega_b}\sum_k \left[\int_{E_{j,k}} \left| \langle  W_{x,a^j},F  \rangle  \right|^2 d\mu(x) -
\sum_k \mu(E_{j,k})\left| \langle  W_{x_{j,k},a^j},F  \rangle  \right|^2\right] \right|,\\
& = \left| \sum_{j < \Omega_b}\sum_k \int_{E_{j,k}}  
\left[\left| \langle  W_{x,a^j},F  \rangle  \right|^2 - \left| \langle  W_{x_{j,k},a^j},F  \rangle  \right|^2\right]d\mu(x) \right|.
\label{iprep}
\end{align}
Next, let ${\cal P}$ denote the usual stereographic projection onto the tangent plane to $S^2$ at ${\bf S}$.
For each $j,k$, select $T_{j,k} \in SO(3)$ with $T_{j,k}x_{j,k} = {\bf S}$, and let
${\cal P}_{j,k} = {\cal P}T_{j,k}$; then ${\cal P}_{j,k}$ could be interpreted as a stereographic
projection onto the tangent plane to $S^2$ at $x_{j,k}$.   For $r_1 > 0$, let ${\cal B}(0,r_1)$
denote the open ball of radius $r_1$ centered at $0$ in $\R^2$.  
For $0 < r < \pi$, ${\cal P}_{j,k}: B(0,r)
\to {\cal B}(0,v(r))$ diffeomorphically, where $v(r) = 2\tan(r/2)$.  Moreover, for some smooth positive function 
$h$ on $R^2$, for any $z_0 \in L^1(B(0,r))$,
\[ \int_{B({\bf S},r)} z_0(x)dx = \int_{{\cal B}(0,v(r))} z_0({\cal P}^{-1}w) h(w) dw. \]
Since rotations preserve the measure on $S^2$, we also have that 
for any $z \in L^1(B(x_{j,k},r))$,
\[ \int_{B(x_{j,k},r)} z(x)dx = \int_{{\cal B}(0,v(r))} z({\cal P}_{j,k}^{-1}w) h(w) dw, \]
Thus

\begin{align}\notag
I & = \left|\sum_{j < \Omega_b}\sum_k \int_{B(x_{j,k},ba^j)} 
\left[\left| \langle  W_{x,a^j},F  \rangle  \right|^2 - \left| \langle  W_{x_{j,k},a^j},F  \rangle  \right|^2\right] 
\chi_{E_{j,k}}(x)dx\right| \\
& = \left|\sum_{j < \Omega_b}\sum_k \int_{{\cal B}(0,v(ba^j))} 
\left[\left| \langle  W_{{\cal P}_{j,k}^{-1}w,a^j},F  \rangle  \right|^2 - 
\left| \langle  W_{x_{j,k},a^j},F  \rangle  \right|^2\right] \chi_{E_{j,k}}({\cal P}_{j,k}^{-1}w)h(w)dw\right| \\
&= \left|\sum_{j < \Omega_b}\sum_k \left(v(ba^j)\right)^n \int_{\mathcal B}
\left[\left| \langle  W_{{\cal P}_{j,k}^{-1}(v(ba^j)w),a^j},F  \rangle \right|^2 
- \left| \langle  W_{x_{j,k},a^j},F  \rangle  \right|^2\right] H_{j,k}(w) dw\right|, \label{iifst}
\end{align}

where ${\cal B} = {\cal B}(0,1)$, and where
\begin{equation}\notag
%\label{hhdf}
H_{j,k}(w) = \chi_{E_{j,k}}({\cal P}_{j,k}^{-1}(v(ba^j)w)h(v(ba^j)w).
\end{equation}
Thus

\begin{equation}
\label{iiscd}
 I = \left|\sum_{j < \Omega_b}\sum_k \mu(E_{j,k}) 
\int_{\mathcal B} 
\left[\left| \langle  W_{{\cal P}_{j,k}^{-1}(v(ba^j)w),a^j},F  \rangle \right|^2 
- \left| \langle  W_{x_{j,k},a^j},F  \rangle  \right|^2\right] G_{j,k}(w) dw\right|,
\end{equation}
where
\[ G_{j,k}(w) = \frac{\left(v(ba^j)\right)^n}{\mu(E_{j,k})} H_{j,k}(w) \]
Note that, by (\ref{measgeq}), there is a constant $C$ such that 
$0 \leq G_{j,k}(w) \leq C$ for all $(j,k)$ with $j < \Omega_b$, and all $w \in 
{\mathcal B}$.  (Recall that $v(r) = 2\tan(r/2)$, so that for any fixed $r_0  > 0$,
there is a $C > 0$ such that $v(r) \leq Cr$ for $0 < r \leq r_0$.)

Applying the fundamental theorem of calculus, we now see that

\begin{equation}
\label{iibtr}
I = \left|\sum_{j < \Omega_b}\sum_k \mu(E_{j,k}) \int_{\mathcal B}
\int_0^1 \frac{\partial}{\partial u}\left|
 \langle  W_{{\cal P}_{j,k}^{-1}(v(ba^j)uw),a^j},F  \rangle \right|^2 du\: G_{j,k}(w) dw\right|.
\end{equation}
For $j < \Omega_b$, $w \in {\mathcal B}$ and $0 \leq u \leq 1$, 
$y \in S^2$, let us set
\begin{equation}\notag
\varphi_{j,k}^{w,u}(y)  = \sqrt{G_{j,k}(w)}W_{{\cal P}_{j,k}^{-1}((v(ba^j)uw),a^j,R_{j,k}}(y)
\end{equation}
(For this to make sense, we must check that $x^{j,k}_{w,u} := {\cal P}_{j,k}^{-1}((v(ba^j)uw) \in U_{R_{j,k}}$. 
But $|v(ba^j)uw| \leq v(ba^j)$, so $x^{j,k}_{w,u} \in B(x_{j,k},ba^j) 
\subseteq B(x_{j,k},\pi/10)$.  Accordingly, if $m = m(j,k)$, then $d(x_{w,u},y_m) < \pi/5$.  
Accordingly $R_m^{-1}x^{j,k}_{w,u} \in B(R_m^{-1}y_m,\pi/5) = B({\bf E},\pi/5) \subseteq U_I$,
so surely $x^{j,k}_{w,u} \in U_{R_m} = U_{R_{j,k}}$.)
By (\ref{iibtr}),
\begin{equation}
\label{iibtr1}
I = \left|\sum_{j < \Omega_b}\sum_k \mu(E_{j,k}) \int_{\mathcal B}
\int_0^1 \frac{\partial}{\partial u}\left|
 \langle \varphi_{j,k}^{w,u},F  \rangle \right|^2 du\: dw\right|.
\end{equation}
Note that for each $m'$ ($1 \leq m' \leq M$), and $y \in U_{R_{m'}}$, we have
\begin{equation}
\label{phiprp}
[\varphi_{j,k}^{w,u}]_{R_{m'}}(y) = \sqrt{G_{j,k}(w)}\overline{K}_{a^j,R_{j,k},R_{m'}}({\cal P}_{j,k}^{-1}(v(ba^j)uw),y).
\end{equation}
We will use the product rule to perform the $u$ differentiation in (\ref{iibtr1}), and so we must examine
\[ \psi_{j,k}^{w,s} := b^{-1}(\partial \varphi_{j,k}^{w,u}/\partial u)\]
so that for each $m'$ ($1 \leq m' \leq M$), and $y \in U_{R_{m'}}$, we have 
\begin{equation}
\label{psiprp}
[\psi_{j,k}^{w,u}]_{R_{m'}}(y) = b^{-1}\sqrt{G_{j,k}(w)}
\frac{\partial}{\partial u}\overline{K}_{a^j,R_{j,k},R_{m'}}({\cal P}_{j,k}^{-1}(v(ba^j)uw),y).
\end{equation}
Write points in $\C$ as $q_1+iq_2$.  We may define smooth vector fields $Z_1, Z_2$ on $S^2\setminus{\bf N}$ by
$Z_1\omega = \partial (\omega \circ {\cal P}^{-1})/\partial q_1$,
$Z_2\omega = \partial (\omega \circ {\cal P}^{-1})/\partial q_2$ (for $\omega \in C^1(S^2\setminus{\bf N}))$.

We recall that ${\cal P}_{j,k}^{-1} = T_{j,k}^{-1} \circ {\cal P}^{-1}$.  Fixing $y$, 
we apply Proposition \ref{rotbdd} with $R = R_{j,k}^{-1}$, $Z = Z_1$ or $Z_2$, $p = {\bf S}$,
$\delta = \pi/10$, and $\omega(x) = \overline{K}_{a^j,R_{j,k},R_{m'}}(x,y)$.  We apply the chain
rule in performing the $u$-derivative in (\ref{psiprp}), and recall again that for any $r_0 > 0$, there
is a $C > 0$ for which $v(r) \leq Cr$ whenever $0 < r \leq r_0$.  We again abbreviate 
$x^{j,k}_{w,u} = {\cal P}_{j,k}^{-1}((v(ba^j)uw)$.  We then see from  
Theorem \ref{locestls} and the definition of ${\cal M}_{x,t}$, that for some $C > 0$,
\[ \varphi_{j,k}^{w,u}, \psi_{j,k}^{w,u} \in C{\mathcal M}_{s,x^{j,k}_{w,u},a^j}, \]
whenever $j < \Omega_b$, $w \in {\mathcal B}$ and $0 \leq u \leq 1$.  However,
\[ d(x^{j,k}_{w,u},x_{j,k}) = v^{-1}(v(ba^j)w) \leq ba^j < a^j, \]
so by (\ref{nrbyok}), for some $C > 0$,
\[ \varphi_{j,k}^{w,u}, \psi_{j,k}^{w,u} \in C{\mathcal M}_{s,x_{j,k},a^j}. \]
 We now find from (\ref{iibtr}) that
\begin{align}\notag
I & = b\sum_{j < \Omega_b}\sum_k \mu(E_{j,k}) \int_{\mathcal B}
\int_0^1 \left[ \langle  \psi_{j,k}^{w,s},F  \rangle   \langle  F, \varphi_{j,k}^{w,s} \rangle  +
 \langle  \varphi_{j,k}^{w,s},F  \rangle   \langle  F,\psi_{j,k}^{w,s} \rangle \right] dsdw\\\notag
& = b \int_{\mathcal B}
\int_0^1 \left[ \langle  S_{\{\varphi_{j,k}^{w,s}\},\{\psi_{j,k}^{w,s}\}}F,F  \rangle 
+  \langle  S_{\{\psi_{j,k}^{w,s}\},\{\varphi_{j,k}^{w,s}\}}F,F  \rangle \right] dsdw\\\notag
& \leq Cb\|F\|^2
\end{align}
by Theorem \ref{sumopthm}.  (The interchange of order of summation and integration
is justified by the dominated convergence theorem and the second sentence of
Theorem \ref{sumopthm} (a).)   This proves (a).

To prove (b), we need only show (\ref{snrtgt}).  But from 
(\ref{qsclose}) and (\ref{strsum1}), we have that, if $F = (I-P)F \in L^2_s$, then
\[ (A_a - C_0b)\|F\|^2 \leq
 \langle  QF, F  \rangle  - C_0b\|F\|^2 \leq
 \langle  SF, F  \rangle  \leq  \langle  QF, F  \rangle  + C_0b\|F\|^2 
\leq (B_a + C_0b)\|F\|^2. \] 
If, on the other hand, $F \in L^2_s$ is general, we have $SF = S(I-P)F$,
since all $\varphi_{j,k} \in (I-P)L^2_s$.  Since
$S$ is self-adjoint, $ \langle  SF, F  \rangle  =  \langle  S(I-P)F, (I-P)F  \rangle $, so in
general
\[ (A_a - C_0b)\|(I-P)F\|^2 \leq  \langle  SF, F  \rangle  \leq 
(B_a + C_0b)\|(I-P)F\|^2, \] 
as desired.
\end{proof}

\section{Proof of the near-diagonal localization result}\label{genf}
In this section we prove Theorem \ref{locestls}.  As we noted, if $f$ has compact support away from 
the origin, this theorem was proved in \cite{GeMari08}.
To handle the case of general $f$, we need to adapt the proof of Lemma 4.1 in \cite{gmcw},
(where we proved near-diagonal localization of similar scalar kernels on general smooth compact
oriented Riemannian manifolds).
That proof, in section 3 of \cite{gmcw}, used two key facts:\\
\ \\
(a) Corollary of finite propagation speed for the wave equation: Say
$g \in \mathcal{S}(\RR)$ is even, and satisfies \supp$\hat{g} \subseteq (-1,1)$, and let
$K_t^g(x,y)$ be the kernel of $g(t\sqrt{\Delta})$.  Then for some $C > 0$, if $d(x,y) > C|t|$,
then $K_t^g(x,y) = 0$.\\
(b) Strichartz's theorem: if $p(\xi) \in S_1^j(\RR)$ is an ordinary symbol of order $j$, depending
only on $\xi$, then $p(\sqrt{\Delta}) \in OPS^j_{1,0}({\bf M})$.\\
\ \\
As one would expect, these facts have analogues in the spin situation, as we now explain.  First, in analogy 
to (a), we 
claim:
\begin{proposition}
\label{finprops}
Say $g \in \mathcal{S}(\RR)$ is even, and satisfies \supp$\hat{g} \subseteq (-1,1)$.  Let
$K_t^g$ be the kernel of $g(t\sqrt{\Delta_s})$.  Then, for some $C > 0$, if $d(x,y) > C|t|$,
then $K_t^g(x,y) = 0$.
\end{proposition}
\begin{proof}
As we shall explain, this follows directly from the finite propagation speed
property for familiar types of wave equations in $\RR^n$, which we now review.  In this
proof all functions and differential operators will be assumed to be smooth, without further comment.

Suppose that $L$ is a second-order differential operator on an open set $V$ in $\RR^n$, that $L$ is elliptic, and
in fact that, for some $c > 0$,  its
principal symbol $\sigma_2(L)(x,\xi) \geq c^2|\xi|^2$, for all $(x,\xi) \in V \times \RR^n$.  
Suppose that $U \subseteq \RR^n$ is open, and that $\overline{U} \subseteq V$.  
Then if $\supp F,G \subseteq K \subseteq U$, where $K$ is compact, then any solution $u$ of
\begin{align}
(\frac{\partial^2}{\partial t^2} + L)u = 0 \\
u(0,x) = F(x)\\
u_t(0,x) = G(x)
\end{align}
on $U$ satisfies $\supp u(t,\cdot) \subseteq \{x: $ dist $(x,K) \leq |t|/c\}$.\\
\ \\
(This is a special case of Theorem 4.5 (iii) of \cite{Tay81}.  In that reference, $V = \RR^n$.  But we can
always extend $L$ from $U$ to an operator on all of $\RR^n$ satisfying the hypotheses, by letting
$L' = \psi L + c^2(1-\psi)\Delta$ for a cutoff function $\psi \in C_c^{\infty}(V)$ which equals $1$
in a neighborhood of $\overline{U}$.)

It is an easy consequence of this that a similar result holds for spin functions on $S^2$.  Let us
look at the problem
\begin{align}
(\frac{\partial^2}{\partial t^2} + \Delta_s) u = 0 \\
u(0,x) = F(x)\\
u_t(0,x) = G(x)
\end{align}
on $S^2$.  The first thing to note is that the problem has a unique solution in any open $t$-interval
about zero.  Namely, if
$F = \sum_{l \geq |s|}\sum_m a_{lm}\:{}_sY_{lm}$,
$G = \sum_{l \geq |s|}\sum_m b_{lm}\:{}_sY_{lm}$,
then the solution is
\begin{equation}
\label{uform}
u(t,x) = \sum_{l \geq |s|}\sum_m
[a_{lm}\cos (\sqrt{\lambda_{ls}} t) + b_{lm} \frac{\sin (\sqrt{\lambda_{ls}} t)}{\sqrt{\lambda_{ls}}}]\:{}_sY_{lm}(x),
\end{equation}
where we interpret $\frac{\sin(\sqrt{\lambda_{ls}} t)}{\sqrt{\lambda_{ls}}}$ as $t$ if $\lambda_{ls} = 0$. 
This follows just as it would in the case $s=0$, from the rapid decay of the $a_{lm}$ and the
$b_{lm}$, and the formal self-adjointness of $\Delta_s$ on $C^{\infty}_s$.
(See Theorem 4.2 and equation (57) of \cite{GeMari08}.)  Note also:
\begin{equation}
\label{cosis}
u = \cos(t\sqrt{\Delta_s})F \mbox{ is the solution if } G \equiv 0.
\end{equation}

We now claim that there is an absolute constant $C > 0$,
such that if $\supp F,G \subseteq K \subseteq S^2$, ($K$ compact), then 
the solution $u$ satisfies  $\supp u(t,\cdot) \subseteq \{x: d(x,K) \leq C|t|\}$.
This is proved as follows:\\
\begin{enumerate}
\item
It is enough to show that, for some $\delta > 0$, the result is true whenever $|t| < \delta$.  For,
suppose that this is known.  It suffices then to show that if, for some $T > 0$, the result is true 
whenever $|t| < T$, then it is also true whenever $|t| < T+\delta$.  For this, say $T \leq t < T + \delta$,
and select $t_0 < T$ with $t-t_0 < \delta$.  
By assumption, $\supp u(t_0,\cdot) \subseteq K' := \{x: d(x,K) \leq Ct_0\}$, and thus also
$\supp u_t(t_0,\cdot) \subseteq K'$.  We clearly have that $u(t,x) = v(t-t_0,x)$, where
$v$ is the solution of 
\begin{align}
(\frac{\partial^2}{\partial t^2} + \Delta_s)v = 0 \\
v(0,x) = u(t_0,x)\\
v_t(0,x) = u_t(t_0,x)
\end{align}
Thus 
\[ \supp u(t,\cdot) = \supp v(t-t_0,\cdot) \subseteq \{x: d(x,K') \leq C(t-t_0)\} 
\subseteq \{x: d(x,K) \leq Ct\}\]
 as claimed.  Similarly if $-T \geq t \geq -T-\delta$.
\item
It suffices to show that, for some $\delta, \epsilon > 0$, the result is true whenever $|t| < \delta$, and
the supports of $F$ and $G$ are both contained in an open ball $B$ of radius $\epsilon$.  For, we could then cover $S^2$
by a finite number of such open balls, and choose a partition of unity $\{\zeta_j\}$ subordinate to this covering.
If we let $(F_j,G_j) = (\zeta_j F, \zeta_j G)$, and if we let $u_j$ be the solution with data $F_j, G_j$ in place
of $F,G$, then surely $u = \sum_j u_j$.  Then surely $\supp u(t,\cdot) \subseteq \{x: d(x,K) \leq C|t|\}$ 
as desired.
\item
To find approprate $\delta, \epsilon$, one need only cover $S^2$ with balls $\{B_k\}$ of some radius $\epsilon$,
for which the balls $B_k'$ with the same centers and radius $2\epsilon$ are charts, 
contained in some $U_R$, and on which, if we use
local coordinates, the geodesic distance is comparable to the Euclidean distance.  The existence of a 
suitable $\delta, C$ now follows at once from the aforementioned result for wave equations in $\RR^2$.
Indeed, on such a ball, it is equivalent to solve
\begin{align}
(\frac{\partial^2}{\partial t^2} + \Delta_{sR}) u_R = 0 \\
u_R(0,x) = F_R(x)\\
(u_R)_t(0,x) = G_R(x)
\end{align}
(Here $\Delta_{sR} = -\bedth_{s+1,R}\edth_{sR}$ if $s \geq 0$, $\Delta_{sR} = -\edth_{s-1,R}\bedth_{sR}$ if $s < 0$.)
Working in local coordinates, we are now clearly in the Euclidean situation.  This proves the ``claim''. 
\end{enumerate}
To prove the proposition, it suffices to write (for some $c$)
\begin{equation}
\label{wavetrck}
g(t{\sqrt \Delta_s})F = c\int_{-1}^{1} \hat{g}(\tau) \cos(\tau t{\sqrt \Delta_s})F d\tau
\end{equation}
for any $F \in C^{\infty}_s$.  (This is easily verified by using the spin sphercial harmonic decomposition
of $F$ and the Fourier inversion formula.)  
The proposition follows at once from (\ref{cosis}) and the ``claim''.
\end{proof}  
\ \\
Next we need to obtain an analogue of Strichartz's theorem ((b) at the beginning of this section).
For this, it is convenient to think of $C^{\infty}_s$ as sections of the line bundle ${\bf L}^s$,
as described in the introduction.  We claim:
\begin{proposition}
\label{strsp}
If $p(\xi) \in S_1^j(\RR)$, then $p(\sqrt{\Delta_s})$ is a pseudodifferential operator of order $j$
(mapping $C^{\infty}({\bf L}^s)$ to $C^{\infty}({\bf L}^s)$).
\end{proposition}
\begin{proof}
One could check to see that Strichartz's proof in \cite{Stric72} goes over to the line
bundle situation.  In our specific situation, however, there is a shortcut: we can reduce the result
to Strichartz's original theorem.  To do this, say $p(\xi) \in S_1^j(\RR)$, and let $T= p(\sqrt{\Delta_s})$,
defined as usual as an operator on $C^{\infty}({\bf L}^s)$ by use of the eigenfunction expansion 
for $\Delta_s$.  It is enough to show that, for every
$R,R' \in SO(3)$,  $\varphi' T \varphi$ is an operator of order $s$, for every
$\varphi \in C_c^{\infty}(U_R)$ and $\varphi' \in C_c^{\infty}(U_{R'})$.  Note also that
$C_c^{\infty}(U_R)$ equals the set of $F_R$, where $F \in C^{\infty}({\bf L}^s)$ has compact
support in $U_R$.  Locally trivializing the bundle ${\bf L}^s$,
is enough to show that the map $S$ which takes $F_R \in C_c^{\infty}(U_R)$
to $(\varphi' T \varphi F)_{R'}$ is in $OPS^j_{1,0}(S^2)$.  We will prove this by
writing $S$ in a different form.

Say $s > 0$.
If $\varphi F = \sum_{l \geq s}\sum_m A_{lm}\:{}_sY_{lm}$, then 
\begin{eqnarray*}
SF_R & = & \varphi'\sum_{l \geq s}\sum_m p(\sqrt{\lambda_{ls}})A_{lm}\:{}_sY_{lmR'}\\
& = & \varphi'\sum_{m=-s}^s p(0)A_{sm}\:{}_sY_{smR'} +
\varphi'\sum_{l \geq s+1}\sum_m p(\sqrt{\lambda_{ls}})A_{lm}\:{}_sY_{lmR'}\\
& := & S_1F_R + S_2F_R,
\end{eqnarray*}
say, and it suffices to prove that both $S_1$ and $S_2$ are in $OPS^j_{1,0}(S^2)$.
Now 
\[ A_{lm} = \langle\varphi F,\:{}_sY_{lm}\rangle = \langle\varphi F_R,\:{}_sY_{lmR}\rangle, \]
so $S_1$ has the smooth kernel $p(0)\sum_{m=-s}^s {}_sY_{smR'}(x)\:\overline{{}_sY_{smR}}(y)$.
Thus we need only consider $S_2$.
Set $\edth^{[s]}_{R'} = \edth_{s-1,R'}\ldots\edth_{0R'}$.  We may write
\begin{eqnarray*} 
S_2F_R & = & \varphi' \edth^{[s]}_{R'}\sum_{l \geq |s|}\frac{p(\sqrt{\lambda_{ls}})}{b_{ls}}A_{lm}Y_{lm}.
\end{eqnarray*}
(Here we have once again used Theorem 4.2 of \cite{GeMari08}, this time to justify the interchange of
differentiation and summation.)  We have that
\begin{equation}
\label{almalt}
A_{lm} = \langle\varphi F,\:{}_sY_{lm}\rangle = \frac{1}{b_{ls}}\langle\varphi F, \edth^s Y_{lm}\rangle
= \frac{1}{b_{ls}}\langle\bedth^s(\varphi F),Y_{lm}\rangle = \frac{1}{b_{ls}}\langle{\cal D}(\varphi F_R),Y_{lm}\rangle,
\end{equation}
where ${\cal D} = \bedth_{1R} \ldots \bedth_{sR}$ is a smooth differential operator of
degree $s$ on $U_R$.  Now, as was noted in \cite{GeMari08}, if we set
\[ \gamma_n = n(n-1), \]
then $(l+n)(l+1-n) = l(l+1)-\gamma_n$ for any $l,n$, so that
\begin{equation}
\label{lslsfac}
b_{ls}^2 =\frac{(l+s)!}{(l-s)!} = [l(l+1)-\gamma_s][l(l+1)-\gamma_{s-1}]\ldots[l(l+1)-\gamma_1].
\end{equation}
while
\begin{equation}
\label{lamls}
\lambda_{ls} = l(l+1)-\gamma_{-s} = l(l+1)-s(s+1).
\end{equation}
Accordingly, 
\begin{equation}
\label{sfrpd}
S_2F_R  =  \varphi' \edth^{[s]}_{R'}\sum_{l \geq s+1}
\frac{p \left(\sqrt{\lambda_l-s(s+1)}\right)}{[\lambda_l-\gamma_s]\ldots[\lambda_l-\gamma_1]}
\langle{\cal D}(\varphi F_R),Y_{lm}\rangle Y_{lm}
\end{equation}

Now select $\zeta \in C^{\infty}(\RR)$ with $\zeta(\xi) = 0$ for $\xi < (s+1)$,
$\zeta(\xi) = 1$ for $\xi > \sqrt{(s+1)(s+\frac{3}{2})}$. Define a function $q$ on $\RR$ by
\[ q(\xi) = \frac{\zeta(\xi) p \left(\sqrt{\xi^2-s(s+1)}\right)}{[\xi^2-\gamma_s]\ldots[\xi^2-\gamma_1]}. \]
Then, as is easily seen, $q \in S_1^{j-2s}(\RR)$.  From (\ref{sfrpd}), we find
\[ S_2F_R = \varphi' \edth^{[s]}_{R'} q(\sqrt{\Delta}) {\cal D}(\varphi F_R) \]
so that, by Strichartz's theorem, $S_2$ is a pseudodifferential operator of order $s + (j-2s) + s = j$, as desired.
Similar arguments work if $s < 0$.
\end{proof}
\ \\
{\em Proof of Theorem \ref{locestls}} Using Propositions \ref{finprops} and \ref{strsp}
the arguments of section 4 of \cite{gmcw} go over, with only very minor changes,
to prove the theorem.  There is a little bit to think about in obtaining analogues of the
two statements marked by $\rhd$, which precede Lemma 4.1 of \cite{gmcw}.  For these 
statements, we
let ${\cal U}'$ (resp. ${\cal U}$) be an open neigborhood of ${\cal F}_{R'}$ (resp. ${\cal F}_R$)
whose closure is contained in $U_{R'}$ (resp. $U_R$).
For the analogues of the two statements in question, 
one uses $K_{\sqrt{t},R',R}(x,y)$ in place of $K_{\sqrt{t}}$, and
uses $({\cal U}' \times {\cal U}) \setminus D$ in place of $({\bf M} \times {\bf M}) \setminus D$ in the statements of
the analogues.  As for their proofs,  
in \cite{gmcw}, we wrote $K_{\sqrt{t}}(x,y) = [f(t\Delta)\delta_y](x)$.  Here we note that, if $r < -n/2$, and if $y \in U_R$,
there is a unique $\Lambda_y \in H^r({\bf L}^s)$ with trivialization $\Lambda_{yR}$ over $U_R$ being $\Lambda_{yR} = \delta_y$;
and that $K_{\sqrt{t},R',R}(x,y) = [f(t\Delta_s)\Lambda_y]_{R'}(x)$.  One then obtains the needed analogues
just as in \cite{gmcw}.

In the
proof of Theorem \ref{locestls}, we may therefore assume $0 < t < 1$ and that $d(x,y) \leq \pi/4$, and consequently that
${\cal F}_{R'}$ and ${\cal F}_R$ are compact subsets of a single $U_{R_0}$.  It is enough to prove the theorem
for $K_{t,R_0,R_0}$ in place of $K_{t,R',R}$, since we are assuming that $0 < t < 1$ and one always has
(\ref{kernel-spin-relation}) for smooth kernel operators.  Now one may proceed entirely analogously to the proof of 
Lemma 4.1 in \cite{gmcw}. 
\vspace{-.3 in}
\flushright$\Box$
\flushleft
\vspace{-.3 in}
\begin{corollary}
\label{cptspt}
In Theorem \ref{framain} (b), for appropriate $f$ and $C > 0$, one has supp$\varphi_{j,k} \subseteq
B(x_{j,k},Ca^j)$.  Thus, for appropriate $a,b$, 
$\left\{ \mu(E_{j,k})^{1/2}\varphi_{j,k}\right\}_{j,k}$ is a nearly tight frame for $(I-P)L^2_s(S^2)$,
with supp$[\mu(E_{j,k})^{1/2}\varphi_{j,k}] \subseteq B(x_{j,k},Ca^j)$.
\end{corollary}
\begin{proof}
 We recall from Theorem \ref{framain} that
$\varphi_{j,k} = W_{x_{j,k},a^j,R_{j,k}}$, and from (\ref{needdf0}) that 
$(W_{xtR})_{R'}(y) = \overline{K}_{t,R,R'}(x,y)$, where as usual $K_t$ is the kernel of
$f(t^2\Delta_s)$.  Suppose that $f(\xi^2) = G(\xi)$ for some even $G \in {\cal S}(\RR)$
with \supp$\hat{G} \subseteq (-1,1)$.  Then $K_t$ is the kernel of $G(t\sqrt{\Delta_s})$,
so the result follows at once from Proposition \ref{finprops}.
\end{proof}


\begin{thebibliography}{99}
\bibitem{spinphys} D. Geller,, F.K. Hansen, G. Kerkyacharian,
D. Marinucci and D. Picard, D. {\it Spin Needlets for Cosmic
Microwave Background Polarization Data Analysis}, Physical Review D,
D78:123533, (2008), arXiv:0811.2881
\bibitem{spinstat} D. Geller, X. Lan and D. Marinucci, {\it Spin Needlets Spectral Estimation}, preprint.
\bibitem{GeMari08} D. Geller, D. Marinucci, {\it Spin Wavelets on the Sphere}, arXiv:0811.2935, submitted. 
\bibitem{gmcw} D. Geller, A. Mayeli, {\em Continuous Wavelets on Compact Manifolds},  Math. Z. {\bf 262} (2009), 895-927.
\bibitem{gmfr}D. Geller, A. Mayeli, {\em Nearly Tight Frames and Space-Frequency Analysis on Compact Manifolds}, to appear in {Math. Z.}, 2009.
\bibitem{gmst}  D. Geller, A. Mayeli, {\em Continuous Wavelets and Frames on Stratified Lie Groups I.}, { J. Fourier Anal. Appl.} 12 (5), 543- 579, 2006.
\bibitem{narc1} F.J. Narcowich, P. Petrushev and J. Ward, \textit{Localized tight frames on spheres},
SIAM J. Math. Anal. {\bf 38} (2006), 574-594.
\bibitem{narc2} F.J. Narcowich, P. Petrushev and J. Ward, \textit{Decomposition of Besov and Triebel-Lizorkin
spaces on the sphere}, J. Func. Anal. {\bf 238} (2006), 530-564.
\bibitem{np} E.T. Newman,  R. Penrose, {\em Notes on the Bondi-Metzner-Sachs Group,}  J. Math. Phys. (7)  1966 
863 -- 870.
\bibitem{stha}  E. Stein, {\em Harmonic Analysis}, Princeton University Press,  1995.
\bibitem{Stric72} R. Strichartz, \textit{A functional calculus for elliptic 
pseudodifferential operators}, Amer. J. Math {\bf 94} (1972), 711-722.
\bibitem{Tay81} M. Taylor, \textit{Pseudodifferential Operators}, Princeton University
Press, 1981.
\end{thebibliography}
\end{document}